\documentclass[12pt, a4paper]{amsart}
\usepackage[top=1.5in, bottom=1.5in, left=1in, right=1in]{geometry}
\usepackage{etoolbox}
\apptocmd{\sloppy}{\hbadness 10000\relax}{}{}

\usepackage[utf8]{inputenc}

\usepackage{comment}
\usepackage{mathrsfs}
\usepackage{latexsym}
\usepackage{graphics}
\usepackage{graphicx}
\usepackage{amsmath,amsfonts,amsthm,amssymb,amscd,amsxtra}
\usepackage{dsfont}

\usepackage{hyperref}
\usepackage[alphabetic]{amsrefs}

\newtheorem{theorem}{Theorem}

\newtheorem{corollary}{Corollary}
\newtheorem{lemma}{Lemma}
\newtheorem{proposition}{Proposition}

\theoremstyle{remark}
\newtheorem*{remark}{Remark}

\newtheorem{problem}{Problem}

\DeclareMathOperator{\rect}{Rect}

\newcommand{\R}{\ensuremath{\mathbb{R}}}

\newcommand{\Z}{\ensuremath{\mathbb{Z}}}

\newcommand{\N}{\mathbb{N}}

\DeclareMathOperator{\prob}{Prob}

\newcommand{\ud}{\mathrm{d}}
\newcommand{\dx}{\mathrm{d}x}

\newcommand{\abs}[1]{\left|#1\right|}
\newcommand{\Cnorm}[1]{\left\|#1\right\|}

\renewcommand{\Re}{\textup{Re }} % \textup prevents Re from being italicized in a thm
\renewcommand{\Im}{\textup{Im }}

\newcommand{\floor}[1]{\left\lfloor#1\right\rfloor}

\newcommand{\frc}[1]{\left\{#1\right\}}
\newcommand{\frcmod}[1]{\left\Vert#1\right\Vert}

\renewcommand{\mod}[1]{{\ifmmode\text{\rm\ (mod~$#1$)}\else\discretionary{}{}{\hbox{ }}\rm(mod~$#1$)\fi}}

\newcommand{\real}{\mathrm{Re}}
\newcommand{\imag}{\mathrm{Im}}

\newcommand{\chibar}{\overline \chi}

\newcommand{\bW}{\mathbb{W}}
\newcommand{\bX}{\mathbb{X}}
\newcommand{\bY}{\mathbb{Y}}
\newcommand{\bm}{\mathbf{m}}
\newcommand{\bn}{\mathbf{n}}

\newcommand{\bE}{\mathbb{E}}

\newcommand{\cP}{\mathcal{P}}
\newcommand{\cS}{\mathcal{S}}
\newcommand{\cB}{\mathcal{B}}

\DeclareMathOperator{\E}{\mathbb{E}}

\title[Mixed character sums and Littlewood polynomials]{Distribution of mixed character sums and extremal problems for Littlewood polynomials}
\author{Jonathan W. Bober}
\author{Oleksiy Klurman}
\author{Besfort Shala}
\address{Heilbronn Institute for Mathematical Research \\ School of Mathematics, University of Bristol, Bristol, United Kingdom}
\email{{\tt j.bober@bristol.ac.uk}}
\address{School of Mathematics, University of Bristol, Bristol, United Kingdom}
\email{{\tt oleksiy.klurman@bristol.ac.uk}}
\address{School of Mathematics, University of Bristol, Bristol, United Kingdom}
\email{{\tt besfort.shala@bristol.ac.uk}}
\begin{document}
\begin{abstract}
    We prove distributional results for mixed character sums \begin{equation*}
\sum_{n\le x }\chi(n)e(n\theta),
\end{equation*}
    for fixed $\theta\in [0,1]$ and random character $\chi \mod q$, as well as for
    a fixed character $\chi$ and randomly sampled $\theta\in  [0,1].$ We present
    various applications of our results. For example, we construct Littlewood
    polynomials with large Mahler measure and $L_1$ norm, thus establishing new records in the
    Mahler and Newman problems. We also prove asymptotics for the $L_{2k}$ norms of well-known Turyn
    polynomials as $k\to\infty$, refuting a conjecture of
    G\"unther and Schmidt. An important ingredient in our work is a general way
    of dealing with ``log-integrability" problems.
\end{abstract}

\maketitle

\thispagestyle{empty}

\section{Introduction}

\newcommand{\Sabt}{S(\chi,\alpha,\beta,\theta)}
In this paper, we develop an approach aiming to understand the mixed
character sums
\begin{equation*}
S(\chi, x, \theta) = \sum_{n\le x }\chi(n)e(n\theta),
\end{equation*}
where $\chi$ is a (multiplicative) Dirichlet character mod $q$ and
$e(x) := e^{2 \pi i x}$ is the additive character on $\R.$ We are motivated by
the following natural questions.
\begin{problem}\label{Q1}
For various ranges of parameters $q,x\to\infty$ determine the distribution of
$S(\chi,x,\theta)$ as $\chi$ runs over a family of Dirichlet characters.
\end{problem}
\begin{problem}\label{Q2}
Let $\chi_q$ be a quadratic character. What is the limiting distribution of
$S(\chi_q,x,\theta)$ when $\theta\in[0,1]$ is chosen uniformly at random and
$q\to\infty$?
\end{problem}

In the case $\theta = 0$, Problem \ref{Q1} simply asks for the understanding
of the partial sums $\sum_{n \le x} \chi(n)$, which is the most studied aspect
of this question. One of the notable results in this regard is that of Harper \cite{harper}, who showed that in the range when both $x$ and $q/x$ tend to infinity with $q$, the sums $S(\chi, x, 0)$ typically exhibit better than square-root cancellation. 

Following Montgomery and Vaughan \cite{MV-meanvalues},
Bober and Goldmakher \cite{BG-maximum} and later Bober, Goldmakher, Granville
and Koukoulopoulos \cite{BGGK} studied the distribution of $\max_{x} \abs{S(\chi,x, 0)}$
as $\chi$ varies over the characters mod $q$. Hussain \cite{hussain-character-paths},
inspired by \cite{kowalski-sawin},
considered the
sum $S(\chi, tq,  0)$ as a function of $t$ and studied its distribution for randomly chosen $\chi$.
Subsequently the more difficult case of quadratic characters was dealt with by
Lamzouri \cite{lamzouri-quadraticmax} and Hussain and Lamzouri
\cite{lamzouri-hussain}, and more recently Lamzouri and Nath
\cite{nath-lamzouri} studied the distribution of the maximum of partial sums of
cubic characters.

In the case when $\theta\neq 0$, Wang and Xu \cite{wang-xu} showed that a smoothed version of $S(\chi, x, \theta)$ is typically of size $\sqrt x$ for all irrational $\theta$ satisfying a weak Diophantine condition, contrasting Harper's result.
Very recently, Dell and Milićević \cite{dell-milicevic} investigated the distribution
of the incomplete Gauss sums $t \rightarrow S(\chi_q, tq, 1/q)$,
where $\chi_q$ is the Jacobi symbol mod $q$ and $q$ is chosen randomly in a dyadic interval, which is
an example of Problem \ref{Q1} with rational $\theta \ne 0.$

In general, the limiting behavior significantly depends on the length $x=x(q)$ of the
character sum, and in this paper we confine ourselves to the fixed range
$\frac{x}{q}\to \beta > 0$ when $q\to\infty.$ To this end,  we consider for a
prime $q$ and a primitive (non-principal) Dirichlet character $\chi$ mod $q$ the shifted mixed
character sum
\[
    \Sabt = \sum_{\alpha q < n \leq (\alpha + \beta) q} \chi(n)e(n\theta).
\]
For some of our applications we will think of this as the polynomial
\[
    \sum_{\alpha q < n \leq (\alpha + \beta) q} \chi(n)x^n
\]
evaluated on the unit circle, thus naturally generalizing the well-known
variants of Fekete ($\alpha = 0$ and $\beta = 1$) and Turyn ($\beta = 1$)
polynomials.

We will take a different approach to these distributional questions, motivated
by the work of Klurman, Munsch and Lamzouri \cite{KLM-fekete} on Fekete
polynomials. For each integer $k\in [0,q-1]$ and Dirichlet character $\chi,$ we
define the function
\[
    F_{k,\chi,\alpha,\beta}(t) = \frac{1}{q^{1/2}}
        \sum_{\alpha q < n \leq (\alpha + \beta) q} \chi(n)e(n(k+t)/q),
\]
a normalized version of the sum $S(\chi, \alpha, \beta, \theta)$, with $\theta = (k+t)/q$ and $t\in [0,1).$ We regard
$F_{k,q,\alpha,\beta}:\widehat{\mathbb{F}}_q\rightarrow \mathscr C[0,1]$ as the (discrete) random
process obtained by choosing $\chi$ uniformly at random from the set of {\it
all} characters mod $q.$ We will consider the limiting behavior of
$F_{k,q,\alpha,\beta}$ as $q \rightarrow \infty$ over the primes, with $k$ either
fixed or tending to infinity as well.

For each integer $k$ we define the random process
\[
    F_{k, \alpha,\beta}(t) = \bX \frac{e(\alpha t)}{2\pi i}
    \sum_{l \in \Z}
    \frac{e(\alpha l) (e(\beta(l + t)) - 1)}
    {l + t}
    \bW(k-l),
\]
where $\bW$ is a Steinhaus random multiplicative function (that is, the random variables $\bW(-1)$ and $\bW(p)$ for primes $p$ are independent with each $\bW(p)$ uniformly distributed on the unit circle, $\bW(-1)$ uniformly distributed in $\{-1, 1\}$, and $\bW(kl) = \bW(k)\bW(l)$ for all integers $k$ and $l$)
and $\bX$ is a random variable uniformly distributed on the unit circle,
independent of $\bW$. As we shall see, when $k$ tends to infinity the effect of multiplicativity
subsides, and so we also define the random process
\[
    F_{\alpha,\beta}(t) = \frac{e(\alpha t)}{2\pi i}
        \sum_{l \in \Z}
    \frac{e(\beta(l + t)) - 1}
    {l + t}
    \bX(l),
\] where the $\bX(l)$ are independent realizations of $\bX$. 
Both of these processes are almost surely continuous. Our first main result is
the following convergence statement in the space of continuous functions $\mathscr{C}[0,1].$
\begin{theorem}\label{thm:random-process}
    Fix real numbers $\alpha$ and $\beta>0$. Then:
\begin{enumerate}
    \item
    For each integer $k$, the sequence of random processes 
    $F_{k,q,\alpha,\beta}$ converges in distribution to $F_{k, \alpha, \beta}$
    as $q\to\infty$ over the primes.
    \item
    For each sequence
    $k_q \rightarrow \infty$ with $k_q = q^{o(1)}$,
    the sequence of random processes 
    $F_{k_q,q,\alpha,\beta}$ converges in distribution to $F_{\alpha, \beta}$
    as $q\to\infty$ over the primes.
    \item
    For each sequence
    $k_q \rightarrow \infty$ with $k_q  = o(\pi(q))$,
    there is a full density subset $\cP$ of the primes
    such that the sequence of random processes 
    $F_{k_q,q,\alpha,\beta}$ converges in distribution to $F_{\alpha, \beta}$
    as $q \in \cP$ tends to infinity.
    \item
    For each irrational
    $\theta$, there is a full density subset $\cP_\theta$ of the primes
    such that the sequence of random processes $F_{\floor{\theta
    q},q,\alpha,\beta}$ converges in distribution to $F_{\alpha,\beta}$ as $q
    \in \cP_\theta$ tends to infinity.
\end{enumerate}
\end{theorem}

We discuss the size restrictions on $k_q$ and the necessity of removing a zero
density subset of primes in Sections \ref{roadmap} and \ref{primeRemSect}.

Applied to Problem \ref{Q1}, we obtain the following corollary, addressing
a question of Wang and Xu \cite{wang-xu}. 

\begin{corollary}
For each irrational $\theta$ and real number $\beta>0$ there is a full density
subset $\cP_\theta$ of the primes such that for every $t_0 \in [0,1]$, when
$\frc{q \theta} \rightarrow t_0$ for a sequence of $q \in \cP_\theta$, the
sequence
\[
    \frac{1}{q^{1/2}}\sum_{n \leq \beta q} \chi(n) e(n \theta)
\xrightarrow{\phantom{dd}d\phantom{dd}} F_{0,\beta}(t_0)
\]
as $\chi$ varies over the characters mod $q$.
\end{corollary}

We now turn to Problem \ref{Q2} and fix $\chi_q$ as the quadratic character mod $q$.
We partition the interval $[0,1]$ into pieces $[\frac{k}{q},\frac{k+1}{q})$ for $k\in [0,q-1]$ and our aim is to show that the behavior of the character sum $S(\chi,\alpha,\beta,\theta)$ in a randomly chosen subinterval is governed by an explicit random process.

For each prime $q$, let
$G_{q,\alpha,\beta}:\{0,1, \ldots, q-1\} \rightarrow \mathscr C[0,1]$ be the random process
$k \mapsto \frac{q^{1/2}}{\tau(\chi_q)}F_{k, \chi_q, \alpha, \beta}(t)$,
where the set $\{0,1, \ldots, q-1\}$ is equipped with the uniform measure.
Finally, we introduce the random process
\[
    G_{\alpha,\beta}(t) = \frac{e(\alpha t)}{2 \pi i}
    \sum_{l \in \Z}
    \frac{e(\alpha l) (e(\beta(l + t)) - 1)}
    {l + t}
    \bY(l),
\]
where $\bY(l)$ are independent random variables uniformly distributed in $\{-1, 1\}$.

\begin{theorem}\label{thm:random-quadratic-process}
    For each real numbers $\alpha$ and $\beta>0$, the random process
    $G_{q,\alpha,\beta}$ converges in distribution to
    $G_{\alpha,\beta}$ as $q\to\infty$ along the primes.
\end{theorem}
 We apply these results in conjunction with some new ideas to make progress on
 old questions related to the extremal properties of Littlewood polynomials.
\subsection{Application to the Mahler problem}
 Our first application  concerns a classical quantity of a
 polynomial $P\in\mathbb{C}[x].$ Recall the $L_{\lambda}$ norm of $P$,
 given, for $\lambda > 0$, by
 \[
     \|P\|_{\lambda} = \left(\int_0^1 |P(e(t))|^{\lambda} \ud t\right)^{1/\lambda}
 \]
 and the Mahler measure of $P$, defined as
\begin{equation}\label{def_mah}
M_0(P)=\exp\left(\int_{0}^1\log |P(e(t))|\ud t\right).
\end{equation}
In 1963, Mahler \cite{mahler} (also mentioned as Problem 10 by Borwein in his
book \cite{B-book}) posed the following problem.
\begin{problem}[Mahler Problem]\label{mah_question}
What is
\[b_m=\limsup_{n\to\infty} \frac{M_0(P_n)}{\sqrt{n}},\]
where the supremum is taken over all polynomials $P_n(x)=\sum_{k=0}^na_kx^k$ with $a_k=\pm 1$ for all $k$?
\end{problem}
Problem \ref{mah_question} is a younger cousin of the infamous Lehmer problem \cite{Lehmer}, asking for integer polynomials with a small Mahler measure.
It is easy to see that $\frac{M_0(P)}{\|P\|_2}<1,$ but even the question of whether $b_m=1$ is wide
open. The conclusion that $b_m=1$ would immediately follow from the
existence of ``ultra-flat" polynomials (Littlewood's conjecture) with coefficients $\pm 1.$ Indeed, if for every $\varepsilon>0$ there is a Littlewood polynomial
$P_n$ satisfying the pointwise bound
\[(1-\varepsilon)\sqrt{n}\le |P_n(z)|\le (1+\varepsilon)\sqrt{n}\]
for $|z|=1,$ then plugging this estimate into \eqref{def_mah} and letting
$\varepsilon\to 0$  yields $b_m=1.$ In particular, this implies that if one
relaxes the condition allowing complex unimodular coefficients $|a_k|=1,$ then
we can achieve the upper bound $1.$ This follows from the existence of complex
ultra-flat polynomials which were famously constructed by Kahane \cite{kahane} using a
probabilistic approach (see also related work of Bombieri and Bourgain
\cite{BB} for a deterministic construction). However, Erd\H{o}s conjectured (Problem 5 in \cite{B-book}) that Littlewood
ultra-flat polynomials do not exist. 

Interestingly, in 1970, Fielding \cite{Field} showed directly (without
appealing to ultra-flatness) that
\[
    \limsup_{n\to\infty} \frac{M_0(P_n)}{\sqrt{n}}=1,
\]
when the supremum is taken over all complex unimodular polynomials of degree $n.$ Beller and Newman \cite{beller-newman} quantitatively refined this by
showing that the expression inside the supremum is $1-O(\frac{\log n}{\sqrt{n}}).$

The Mahler problem, which may be regarded as a ``flatness" question ---
albeit in a somewhat weaker form --- has attracted considerable attention over
the last few decades. Choi and Erd\'{e}lyi \cite{CE} proved that for each $n$,
there exists a Littlewood polynomial $f_n$ satisfying $M_0(f_n)/\sqrt{n}\ge
1/2+o(1)$ and consequently $b_m\ge 1/2.$ In \cite{CE1} they established a
stronger result, determining the expected value of the normalized Mahler
measure for random $\pm 1$  Littlewood polynomials, namely
\[
\lim_{n\to\infty}\frac{\mathbb{E}M_0(f_n)}{\sqrt{n}}=e^{-\gamma/2} = 0.749\ldots,
\]
where $\gamma = 0.57721\ldots$ denotes Euler’s constant. Building on the work of Rodgers \cite{Rod}, Erd\'{e}lyi \cite{Rudin-Shapiro} showed
that for the Rudin–Shapiro polynomials $P_k(x)$ and $Q_k(x)$, the value of the
normalized Mahler measure approaches 
$\sqrt{\frac{2}{e}}\approx 0.8577\dots$ when $k\to\infty.$
This provided a record value in the Mahler problem.

Very recently, a new approach to the Mahler question was introduced, relying on
the solution in \cite{KLM-fekete} of an old problem determining the Mahler measure
of Fekete polynomials $F_q(z).$ Previously, using subharmonic methods, Erd\'elyi and
Lubinsky~\cite{LE} proved the lower bound $M_0(F_q)\ge
(\frac{1}{2}-\varepsilon)\sqrt{q},$ which was improved in~\cite{Lower-Mah}
to $M_0(F_q)\ge (\frac{1}{2}+c_1)\sqrt{q},$ for some small value of $c_1>0.$ In
\cite{KLM-fekete}, the authors showed that $M_0(F_q)\sim c\sqrt{q},$ where
$c=0.7408\dots$, by introducing new probabilistic ideas.
Mossinghoff \cite{mossinghoff-mahler} used these techniques to compute the Mahler measure of Turyn
polynomials (corresponding to the shift $\alpha=\frac{1}{4} $ and $\beta=1$)
to significantly improve the record value to $b_m\ge 0.951\dots.$ We note that
these works crucially relied on the analysis of {\it complete exponential
sums} in both the probabilistic and ``log-integrability" parts of the arguments
(see Section \ref{roadmap} for a more detailed discussion).  

We instead apply our distributional results for incomplete exponential sums
(Theorem \ref{thm:random-quadratic-process}) in conjunction with a new
general method of dealing with the log-integrability problem to show the
following result, improving the record in the Mahler problem.
\begin{corollary}\label{mahler-quant}
    There exist Littlewood polynomials of arbitrarily large degree
    with normalized Mahler measure $> 0.954$. Consequently, $b_m>0.954$.
\end{corollary}

The last result follows, after some straightforward adjustments, from the following statement. 
\begin{theorem}\label{thm5} We have
    \[
        \frac{1}{(1.1q)^{1/2}}\exp\left(\int_0^1 \log \abs{S(\chi_q, 0.2, 1.1, \theta)}\ud\theta\right)
             \longrightarrow c\approx 0.954\ldots
    \]
as $q \rightarrow \infty$ over the primes.
\end{theorem}
The values of $\alpha$ and $\beta$ come from an approximate numerical
calculation, and we do not wish to give the impression that the Mahler
measure is maximized at these precise values.
The flexibility of our methods, however, offers a possible way to modify
the construction of the polynomials further, obtaining better bounds. We will explore this in
future work.

In fact, Mossinghoff \cite{mossinghoff-mahler} already conjectured that a larger normalized Mahler measure might
be achievable by considering $\beta \ne 1$, in analogy with the result of Jedwab,
Katz and Schmidt \cite{jedwab} that the merit factor of the generalized Turyn
polynomial is maximized at $\alpha \approx 0.221$ and $\beta = 1.058$. Using  methods of this paper one can show that for these values
\[
   \frac{1}{(1.058q)^{1/2}} \exp\left(\int_0^1 \log \abs{S(\chi_q, 0.221, 1.058, \theta)} \ud \theta\right) \approx 0.9535,
\]
which also provides an improvement over the bound in \cite{mossinghoff-mahler} and
shows that for the generalized Turyn polynomials, the merit factor and the
normalized Mahler measure are not maximized at the same point.

Similar extremal problems have been considered for other norms, most notably
for the $L_1$ norm, which is known as Newman's problem. Mossinghoff records the
new lower bound of $0.9775\ldots$ for the ratio $||P_n||_1/\sqrt{n}$ for an
infinite family of Littlewood polynomials. Following our methods, this is
improved slightly to $0.9783\ldots$ by the generalized Turyn polynomials with
parameters $\alpha = 0.2$ and $\beta=1.1$.

\subsection{Norms of Turyn polynomials}
\phantom{M}In 2017, G\"{u}nther and Schmidt \cite{GuentherSchmidt2017} determined 
 the limiting value of the normalized $L_{2k}$ norm of the Fekete and Turyn
polynomials when $k\in\mathbb{N}$. More precisely, they proved that for each $k\in\mathbb N$, there is a function
$\phi_k:\mathbb{R}\to\mathbb{R}$ such that
$\lim_{q\to\infty}\frac{1}{{\sqrt{q}}}\|S(\chi_q, \alpha, 1, \cdot)\|_{2k}
=\phi_k(\alpha) $ and described a method
to compute $\phi_k(0)$ efficiently. The expressions for $\phi_k(\alpha)$ are given by
rather complicated recursive combinatorial identities, which make them hard to
use in practice. However, they showed that $\phi_k(\alpha)$ attains its minimum
at $\alpha=\frac{1}{4}$ for $k=2,3,4$ and conjectured that this continues to
hold for every $k\ge 5.$ In fact, this conjecture is already false for $k=8$;
following G\"unther and Schmidt it is possible to compute $\phi_8(\alpha)^{16}$
exactly as a piecewise-polynomial and verify that the minimum occurs around
$\alpha \approx 0.241$. We discuss this more in Section \ref{section:GS}.

As an application of our results, we study the limiting behaviour of
$\phi_k(\alpha)$ as $k$ tends to infinity.

\begin{theorem}\label{GSP_extremal}
    For every real $k > 0$ the limit
    \[
        \phi_k(\alpha) :=
        \lim_{q\to\infty}\frac{1}{{\sqrt{q}}}\|S(\chi_q, \alpha, 1, \cdot)\|_{2k}
    \]
    exists and satisfies the following properties.
\begin{itemize}
    \item[(i)] For any positive integer $k$, $\phi_k(\alpha)$ is maximized at $\alpha = 0$.
    \item[(ii)] For any rational number $\alpha = \frac{a}{b}$ (in lowest terms), we have
\[
    \lim_{k \rightarrow \infty} \frac{\phi_k(a/b)}{\log k} = \frac{2}{\pi b}
    \begin{cases}
        \csc\left(\frac{\pi}{2b}\right) & \textrm{ if } $b$ \textrm{ is odd,} \\
        2\csc\left(\frac{\pi}{b}\right) & \textrm{ if } $b$ \textrm{ is even.}
    \end{cases}
\]
\item[(iii)] For any irrational number $\alpha$, we have
\[
    \lim_{k \rightarrow \infty} \frac{\phi_k(\alpha)}{\log k} = \frac{4}{\pi^2}.
\]
\end{itemize}
\end{theorem}

Note that this theorem implies that the conjecture of G\"unther and Schmidt is
false in a strong sense: $\phi_k(1/4) \sim \frac{\sqrt{2}}{\pi} \log k $ is never a
minimum of $\phi_k(\alpha)$ for large enough $k$. It seems likely that
$\phi_k(1/4)$ is never a minimum for $k \ge 8$ and plausible that $\lim_{k
\rightarrow \infty}
\min_{\alpha}\frac{\phi_k(\alpha)}{\log k} = \frac{4}{\pi^2}$.

\section{Roadmap to the proofs}\label{roadmap}
We now briefly describe some key steps in our proofs,
and at the end of this section indicate more precisely how Theorems \ref{thm:random-process} and \ref{thm:random-quadratic-process}
are proved.
The starting point in the proof of both Theorem \ref{thm:random-process} and
Theorem \ref{thm:random-quadratic-process} is an application of the quantitative  Poisson
summation formula (see Proposition \ref{prop:poisson}) which transforms
exponential sums into conditionally convergent series
\begin{equation}\label{intro_trunc}
   \frac{e(\alpha t)\tau(\chi)}{2\pi i q^{1/2}} \sum_{l \in \Z}
    \frac{e(\alpha l) (e(\beta(l + t)) - 1)}
    {l + t }
    \chibar(k-l).
\end{equation}
From this formula we begin to see how the random processes that appear in
Theorems \ref{thm:random-process} and \ref{thm:random-quadratic-process}
emerge. Indeed, for most characters, the sum \eqref{intro_trunc} is well
approximated by a short sum, and so via the method of moments we will establish convergence to the random process
\[
    F_{k, \alpha,\beta}(t) = \bX \frac{e(\alpha t)}{2\pi i}
    \sum_{l \in \Z}
    \frac{e(\alpha l) (e(\beta(l + t)) - 1)}
    {l + t}
    \bW(k-l).
\]
As we shall see, for large $k=k_q\to\infty$, multiplicativity generally becomes unimportant at the cost of removing a
zero density subset of primes, so that the limiting process becomes
\[
    F_{\alpha,\beta}(t) = \frac{e(\alpha t)}{2\pi i}
    \sum_{l \in \Z}
    \frac{e(\beta(l + t)) - 1}
    {l + t}
    \bX(l).
\]
Here the $\bX$ corresponding to the Gauss sum and the $e(\alpha l)$
terms have gone away, as they do not change the random process.

In Theorem \ref{thm:random-quadratic-process}, when we fix the quadratic
character and choose $k$ uniformly at random, we follow similar arguments, but
now we exploit randomness coming from the shifts $\{\chi(k-l)\}_{l\le L(q)}$ to arrive at the limiting process 
\[
    G_{\alpha,\beta}(t) = \frac{e(\alpha t)}{2 \pi i}
    \sum_{l \in \Z}
    \frac{e(\alpha l) (e(\beta(l + t)) - 1)}
    {l + t}
    \bY(l).
\]

An interesting feature of Theorem \ref{thm:random-process} is the presence of the exceptional  set of primes to facilitate convergence of the sequence of random processes $F_{k, q,\alpha,\beta}.$ 

Roughly speaking, when analyzing moments of the distribution
$F_{k,q, \alpha, \beta}$, we are led to consider solutions to the polynomial
equations
\begin{equation}\label{eq:modp-condition}
    \prod_{i=1}^d (k + m_i) \equiv \prod_{j=1}^d (k + n_j) \mod q
\end{equation}
where $k = k_q$ and the variables $m_i$ and $n_j$ are
small (relative to $q$). One would like to argue that the mod
$q$ condition becomes irrelevant as $q$ gets large, so that the moments of \eqref{intro_trunc} coincide with the random multiplicative model $F_{k, \alpha, \beta}$. This is indeed what happens for fixed $k$. The mod $q$ condition also becomes irrelevant when $k_q = q^{o(1)}$, and in this case the moments of \eqref{intro_trunc} coincide with the random model $F_{\alpha, \beta}$ without multiplicativity. However, for certain values of $k=k_q\to\infty$, this need not be the case; if $k \equiv (q+1)/3 \pmod q$ (or $k = \floor{q\theta}$ is very close to $(q+1)/3$ for irrational $\theta$), say,  then this
equation becomes
\[
    \prod_{i=1}^d (1 + 3m_i) \equiv \prod_{j=1}^d (1 + 3n_j) \mod q
\]
and the mod $q$ condition does indeed become superfluous by size
considerations, but this shows that the mod $q$ condition was an
essential feature of the \emph{first} equation \eqref{eq:modp-condition}. In the case $k_q = \floor{q\theta}$,
such linear coincidences happen only finitely often for any irrational
$\theta$, however, so do not present a problem.

More insidious are the coincidences of a higher degree. \phantom{.}For any fixed choice
of $m_1, \ldots, m_d$ and $n_1, \ldots, n_d$, the equation
\eqref{eq:modp-condition} becomes a polynomial equation in $k$ of degree at most $d-1$, so we naturally arrive at
the question of how often is $k = k_q$ the solution of
polynomial equations with small coefficients mod $q.$ In Lemma \ref{lemma:fpalpha}, we show that when $k_q = o(\pi(q))$ or $k_q = \floor{q\theta}$ for irrational $\theta$, this happens for $o(\pi(X))$ choices of primes $q\le X$. We then use iterative density arguments to construct a density one set of suitable primes $\cP_{\theta}$ to avoid such coincidences for higher moments when $q\to\infty.$ 

We now describe the key ideas to prove Theorem \ref{thm5}.
We write the logarithmic Mahler measure of $S_q(\theta) := S(\chi_q, \alpha, \beta, \theta)$ as
\begin{equation}
    \log M_0(S_q)=\int_0^1\log \abs{q^{1/2} \frac{S_q(\theta)}{q^{1/2}}} \ud \theta
    = \log \sqrt{q} + \frac 1q \sum_{k=0}^{q-1} 
    \int_0^1 \log \bigg| F_{k, \chi_q,\alpha,\beta}(t) \bigg| \ud t. \label{eq:Computation}
\end{equation}
We will eventually show that
\[
    \frac 1q \sum_{k=0}^{q-1} 
    \int_0^1 \log \bigg| F_{k, \chi_q,\alpha,\beta}(t) \bigg| \ud t
        \longrightarrow 
        \E\left[\int_0^1 \log\abs{G_{\alpha,\beta}(t)} \ud t\right].
\]
Note, however, that the functional $\ell(f)=\int_{0}^1\log |f(t)|dt$ is not
continuous on $\mathscr C[0,1]$ and so we cannot immediately apply Theorem
\ref{thm:random-quadratic-process}.
We thus have to deal with the issues of uniform log-integrability, which is
a difficult problem for both random and deterministic series.
To this end, we consider the regularized continuous functional
$$\ell_{\varepsilon}(f)=\int_{0}^{1}\log \abs{f(t)}w_{\varepsilon}(\abs{f(t)})\ud t$$ with $w_{\varepsilon}(t)$ a smooth minorant of the indicator function $\mathds 1_{t\geq\varepsilon}$, to which our convergence results of Theorem
\ref{thm:random-quadratic-process} apply more easily.

Our problem then reduces to estimating the logarithmic integrals $$\int_0^1
\log\lvert f(t)\rvert\mathds1_{\lvert f(t)\rvert<\varepsilon}\text{d} t$$
when $\varepsilon$
is small and $f$ runs over all functions of the form \eqref{intro_trunc}. We
note that neither the series of $f$ nor any of its higher order derivatives
converge absolutely (unlike in the case of Fekete and Turyn polynomials)
and so new ideas are needed to prove uniform bounds. To overcome this significant obstacle, we prove the following soft general result,
possibly of independent interest (which we shall use in future work).
\begin{proposition}
\label{prop:intro_logThm}
    Let $\Delta = \sum_{i=0}^k a_i \partial^i$ be a linear differential operator of order $k\geq 1$ with constant real coefficients, and suppose that $(\Delta f)(t)\geq 1$ for all $t\in[\alpha, \beta]$ for real-valued $f\in\mathscr C^{k}[\alpha,\beta].$ Then for small enough $\varepsilon>0$, we have
    \begin{equation*}
    \int_{\alpha}^{\beta} \log\lvert f(t)\rvert \mathds1_{\lvert f(t)\rvert<\varepsilon}\ud t \ll \varepsilon^{\frac{1}{2k}},
    \end{equation*}
    where the implied constant is uniform in $f$, but may depend on $\Delta, \alpha$ and $\beta$. 
\end{proposition}
With this in hand, we cook up a suitable linear differential operator $\Delta$
and cover $[0,1]$ by a number of {\it fixed} intervals, such that
 either $|\Re(\Delta f)| \ge 1$ or
$ |\Im(\Delta f)|\ge 1$ across each interval and then apply Proposition
\ref{prop:intro_logThm} to reach the conclusion.
 
Finally, in order to prove Theorem \ref{GSP_extremal}, we use our Theorem
\ref{thm:random-quadratic-process} to conveniently rewrite the $L_{2k}$ norms and then asymptotically evaluate the resulting expression as $k\to\infty$, reducing it to a maximization problem involving a complete exponential sum. 

\subsection{Proofs of Theorems 1 and 2}

We use the approach of Prokhorov to prove relative compactness of our
sequences of random processes. This will allow us to determine the
existence of limiting distributions by computing and matching moments of
finite-dimensional distributions.

\begin{proposition}\label{prop:prokhorov-criteria}\cite{prokhorov}*{Theorem 2.1}
    Let $L_k(t)$ be a sequence of $\mathscr C[0,1]$-valued random processes.
    If there exist some constants $A > 0$, $B > 0$ and $C > 0$
    such that
    \begin{equation}\label{kolmogorov-condition}
        \E \abs{L_k(s) - L_k(t)}^A \leq C\abs{s - t}^{1 + B}
    \end{equation}
    for all (sufficiently large) $k$ and for all $s, t\in [0, 1]$, and if the
    finite-dimensional distributions of $L_k(t)$ tend to some limits, then $L_k(t)$
    converges weakly to some random process $L$.
\end{proposition}

Instead of working directly with the random processes $F_{k,q,\alpha,\beta}$ and
$G_{q,\alpha,\beta}$, we define truncated versions $\tilde F_{k,q,\alpha,\beta}$
and $\tilde G_{q,\alpha,\beta}$ and prove in Proposition \ref{proposition:truncation}
that a sequence of these truncated random processes converges if and only if the
untruncated processes converge. This is a technical convenience, but it is also
somewhat necessary; the moments of the untruncated random processes will not
converge without removal of the trivial character and moreover, the size of the truncation parameter will play a role in our counting arguments handling various ranges of $k_q.$

In Lemma \ref{lemma:tightness} we show that $\tilde F_{k,q,\alpha,\beta}$ and
$\tilde G_{q,\alpha,\beta}$ satisfy the tightness condition
\eqref{kolmogorov-condition}. (Though we do not need this fact, from this combined
with the moment bounds of Section \ref{sec:moments} we may conclude
that the sequence $\tilde F_{k_q,q, \alpha,\beta}$ is relatively compact for any $k_q$.)

Finally, to apply Proposition \ref{prop:prokhorov-criteria}, we prove that
all mixed moments of all finite-dimensional distributions of
$\tilde F_{k,q,\alpha,\beta}$ tend to those of $F_{k,\alpha,\beta}$ when
$k$ is fixed (Proposition \ref{prop:chi-moments}) or tend to those of
$F_{\alpha,\beta}$ when $k$ tends to infinity, after possibly removing some primes (Proposition \ref{prop:chi-moments2}),
in each of the cases (2), (3), and (4) of Theorem \ref{thm:random-process}. These propositions also include bounds
for the moments which show that they determine the distribution (Carleman's condition),
yielding a proof of Theorem \ref{thm:random-process}. Similarly, we deal with the
quadratic case in Proposition \ref{prop:quadratic-moments}, completing the
proof of Theorem \ref{thm:random-quadratic-process}.

\section{The random processes and truncation}

We begin by recalling the ``twisted'' Poisson summation formula.
\begin{lemma}\label{lemma:poisson}
    For a primitive Dirichlet character $\chi$ mod $q$ and $f \in L^1(\R)$ of
    bounded variation
    \[
        \sum_{n \in \Z} \tilde f(n) \chi(n)
            = \frac{1}{q} \tau(\chi) \sum_{k \in \Z} \hat f\left(\frac{k}{q}\right) \chibar(k),
    \]
    where
    \[
    \tilde f(n) = \frac{1}{2}\lim_{\varepsilon \rightarrow 0}\big(f(n + \varepsilon) + f(n-\varepsilon)\big)
    \]
    and $\tau(\chi)$ denotes the Gauss sum of $\chi$. 
\end{lemma}
\begin{proof}
    A version of this is, for example, Theorem 10.5 of \cite{kookoobook};
    however there it is stated in a less general form. For the more general
    version, one could refer to the Poisson summation formula of
    \cite{MV}*{Theorem D.3} and insert it into the proof of
    \cite{kookoobook}*{Theorem 10.5}.
\end{proof}

We will also use the following quantitative bound for convergence of Fourier
series.
\begin{lemma}[\cite{MV}*{Theorem D.2}]\label{lemma:fourierbound}
    Suppose that $f$ is a function of bounded variation on $[0,1]$. Then for
    any $\theta$,
\[
    \abs{\tilde f(\theta) - \sum_{\abs{n} \le K}\hat f(n) e(n\theta)}
    \le \int_0^1 \min\left(\frac{1}{2}, \frac{1}{(2K + 1)\pi \sin \pi x}\right) \abs{\ud f(\theta + x)}.
\]
\end{lemma}

Equipped with these lemmas we are in a position to state and prove our general
formula for incomplete shifted mixed character sums.

\begin{proposition}\label{prop:poisson}
    For a primitive Dirichlet character $\chi$ mod $q$ and for $\alpha,\beta \in \R$,
    we have
\begin{multline*}
    F_{k,\chi,\alpha,\beta}(t)
    = e(\alpha t)\frac{\tau(\chi)}{2\pi i q^{1/2}}\sum_{\abs{l} < K}
    \frac{e(\alpha l) (e(\beta(l + t)) - 1)}
    {l + t }
    \chibar(k-l)
    \\
    +
    O\left(\frac{q^{1/2} \log(q)}{K} + \frac{1}{q^{1/2}}\min\left(1, \frac{q}{K \frcmod{\alpha q}}\right)
        + \frac{1}{q^{1/2}}\min\left(1, \frac{q}{K\frcmod{(\alpha + \beta)q}}\right)
    + q^{1/2}\frac{\log K}{K}
    \right).
\end{multline*}
    In particular, for any $\alpha$ and $\beta$ such that $\alpha q$ and
    $(\alpha + \beta)q$ are not integers, we have
    \begin{equation}\label{elegant_formula}
    F_{k,\chi,\alpha,\beta}(t) =
    e(\alpha t)\frac{\tau(\chi)}{2\pi i q^{1/2}}
    \lim_{K \rightarrow \infty}
    \sum_{\abs{l} < K}
    \frac{e(\alpha l) (e(\beta(l + t)) - 1)}
    {l + t }
    \chibar(k-l).
\end{equation}
\end{proposition}
Note that for $l=0$ and $t=0$ the summand becomes $2\pi i \beta$; for
$l=-1$ and $t=1$ it becomes $2 \pi i \beta e(-\alpha)$.
\begin{proof}
    We apply the Poisson summation formula (Lemma \ref{lemma:poisson})
    to the function
    \[
        f(x) = e(x\theta) \rect((x-\alpha q)/(\beta q)-1/2),
    \]
    where $\rect(u)$ is equal to $1$ if $\lvert u\rvert<\frac{1}{2}$, equal to $0$ if $\lvert u \rvert > \frac{1}{2}$, and equal to $0$ if $\lvert u\rvert = \frac{1}{2}$.
    We find that
    \[
        \sideset{}{'}\sum_{\alpha q \le n \le \alpha q + \beta q}\chi(n)e(n\theta)
         = e(\alpha q\theta) \frac{\tau(\chi)}{2\pi i} \sum_{m \in \Z}
        \frac{e(-\alpha m) (e(\beta(\theta q - m)) - 1)}
        {\theta q - m}
        \chibar(m),
    \]
    where $\sum'$ indicates that if $\alpha q$ or $(\alpha+\beta)q$ is an integer, then it is counted with a $\frac{1}{2}$-weight. 
Now write $\theta = (k + t)/q$, with $k = \floor{q\theta}$ and $t = \frc{q \theta}$.
Then
\[
\begin{split}
    e(-\alpha q \theta)
    \sideset{}{'}\sum_{\alpha q \le n \le \alpha q + \beta q} \chi(n)e(n\theta)
    =& \frac{\tau(\chi)}{2\pi i} \sum_{m \in \Z}
    \frac{e(-\alpha m) (e(\beta(k + t - m)) - 1)}
    {k + t - m}
    \chibar(m) \\
    =& \frac{\tau(\chi)}{2\pi i} \sum_{l \in \Z}
    \frac{e(-\alpha (k-l)) (e(\beta(l + t)) - 1)}
    {l + t }
    \chibar(k-l).
\end{split}
\]
Hence
\[
    \frac{2 \pi i e(-\alpha\frc{q\theta})}{\tau(\chi)}
    \sideset{}{'}\sum_{\alpha q \le n \le \alpha q + \beta q} \chi(n)e(n\theta)
    = \sum_{l \in \Z}
    \frac{e(\alpha l) (e(\beta(l + t)) - 1)}
    {l + t }
    \chibar(k-l).
\]
Let $g(\alpha)$ denote the left hand side, which, for the moment, we think of
as a function of $\alpha$. This is periodic on $[0,1]$ and the right hand side
is a Fourier series. The derivative of $g$ away from the endpoint jumps is $\ll q^{1/2}$, whereas the endpoint jumps each have $\ll \frac{1}{\sqrt q}$ contribution, therefore
\begin{multline*}
    \int_0^1 \hspace{-5pt}\min\hspace{-2pt}\bigg(\frac{1}{2}, \frac{1}{(2K + 1)\pi \sin \pi x}\bigg) |\ud g(\alpha + x)|
    \ll
    {q^{1/2}}\int_0^1 \hspace{-5pt}\min\hspace{-2pt}\left(\frac{1}{2}, \frac{1}{(2K + 1)\pi \sin \pi x}\right) \dx \\
    + \frac{1}{\sqrt q}\sum_{j=0}^{q-1}\min\left(\frac{1}{2}, \frac{1}{(2K + 1)\pi \sin \pi x_j}\right) 
+ \frac{1}{\sqrt q}\sum_{j=0}^{q-1}\min\left(\frac{1}{2}, \frac{1}{(2K + 1)\pi \sin \pi y_j}\right),
\end{multline*}
where
\[x_j = \left\{\frac{j}{q} - \alpha\right\} \ \ \ \textrm{ and }\ \ \ 
y_j = \left\{\frac{j}{q} - (\alpha + \beta)\right\}.\]
We can bound the latter two sums by
\[
    O\left(\frac{q \log(q)}{K} + \min\left(1, \frac{q}{K\frcmod{\alpha q}}\right)
    + \min\left(1, \frac{q}{K\frcmod{(\alpha + \beta)q}}\right)
    \right)
\]
and the integral by $O((\log K)/K)$, so that Lemma \ref{lemma:fourierbound}
gives us the quantitative bound claimed.
\end{proof}

\begin{remark}
    The novelty of this proposition is the use of the Poisson summation
    formula to treat {\it all} values of $\beta.$ Interestingly, the identity \eqref{elegant_formula} with $\alpha=0$ and
    $\beta=1$ is implicit in the work of Montgomery which studied the
    maximum of the complete sum for quadratic characters. In this case, it is
    \cite{montgomery-fekete}*{Lemma 1, Page 376} combined with an elegant
    observation about the cotangent sum that Montgomery makes on the next page
    and the trigonometric identity
    \begin{equation*}
        e((k + t)/2)(\sin \pi (k+t)) = \frac{1}{2i}(e(t) - 1).
    \end{equation*}
    This leads to a simpler formula
    \[
        S\left(\chi, 0, 1, \frac{k+t}{q}\right) =
        \frac{\tau(\chi)}{2\pi i}(e(t)-1)
        \lim_{N \rightarrow \infty}\sum_{n=-N}^N \frac{\chibar(k + n)}{t - n}.
    \]
    Variants of this identity, again for complete sums ($\beta=1$) also occurred in \cite{CGPS-fekete},
    \cite{KLM-fekete} and \cite{mossinghoff-mahler}; however, the authors did not seem to notice Montgomery's
    cotangent sum identity, and instead used properties of Gauss sums to deduce various approximations.

\end{remark}

We will first prove that we can truncate the infinite sum over $l$ when
we vary either the character $\chi$ or the shift $k$. We choose some slowly growing function $L(q)$
and for each $k$ and $\chi$, we define the function
\[
    \tilde F_{k,\chi,\alpha,\beta}(t)
    = e(\alpha t)\frac{\tau(\chi)}{2\pi i q^{1/2}}\sum_{\abs{l} \leq L(q)}
    \frac{e(\alpha l) (e(\beta(l + t)) - 1)}
    {l + t}
    \chibar(k-l).
\]
For each $q$ we consider the random process $\tilde F_{k,q,\alpha,\beta}$
which comes from choosing a character $\chi$ mod $q$ uniformly at random,
and the random process $\tilde G_{q, \alpha, \beta}$ that comes from
fixing the quadratic character mod $q$ and choosing $k$ uniformly at random,
again adjusting $\tilde G_{q,\alpha,\beta}$ to remove the contribution of the
Gauss sum.
In most of what follows, the only important properties of $L(q)$ are
that it tends to infinity and that it does so more slowly than $q^{\varepsilon}$ for
any fixed $\varepsilon$. For our applications it will be necessary to take
$L(q)$ growing arbitrarily slowly.
\begin{proposition}\label{proposition:truncation}
    Let $L(q)$ be some function such that $L(q) \rightarrow \infty$ and
    $L(q) \ll q^\varepsilon$ for all $\varepsilon > 0$.
    For fixed real numbers $\alpha$ and $\beta$, any bounded Lipschitz function $h:\mathscr{C}[0,1] \rightarrow \R$, and
    any sequence of integers $k_q$, as $q$ tends to infinity over the primes we have 
    \[
        \abs{\E h(F_{k_q,q,\alpha,\beta}) - \E h(\tilde F_{k_q,q,\alpha,\beta})} \rightarrow 0 \text{ and } \abs{\E h(G_{q,\alpha,\beta}) - \E h(\tilde G_{q,\alpha,\beta})} \rightarrow 0.
    \]
    That is, $F_{k_q,q,\alpha,\beta}$ converges weakly to some $F$ (as $q$ tends
    to infinity over the primes, or over some subsequence of the primes) if and
    only if $\tilde F_{k_q,q,\alpha,\beta}$ converges weakly to $F$, and similarly for $G_{q, \alpha, \beta}$ and $\tilde G_{q, \alpha, \beta}$. 
\end{proposition}
\begin{proof}
    We will abbreviate $F_{k,\chi} := F_{k,\chi,\alpha,\beta}$.
    By the triangle inequality
    \begin{multline*}
        \abs{F_{k,\chi}(t) - \tilde F_{k,\chi}(t)}^2
        \ll \abs{\tilde F_{k,\chi}(t) -
            e(\alpha t)\frac{\tau(\chi)}{2 \pi i q^{1/2}} \sum_{\abs{n} < q}
    \frac{e(\alpha n)(e(\beta(n + t)) - 1)}
        {t + n}\chibar(k - n)
            }^2 \\
            +
        \abs{F_{k,\chi}(t) -
            e(\alpha t) \frac{\tau(\chi)}{2 \pi i q^{1/2}} \sum_{\abs{n} < q}
    \frac{e(\alpha n)(e(\beta(n + t)) - 1)}
        {t + n}\chibar(k - n)
            }^2.
    \end{multline*}
    The second term is $O((\log q)^2/q)$ by the quantitative bound of
    Proposition \ref{prop:poisson} for any nontrivial character $\chi$ mod $q$,
    so
    \[
    \abs{F_{k,\chi}(t) - \tilde F_{k,\chi}(t)}^2
        \ll
            \abs{
            \sum_{L(q) < |n| < q}
            \frac{e(\alpha n)(e(\beta(n + t)) - 1)}
            {t + n}
            \chibar(k - n)
        }^2 + \frac{(\log q)^2}{q}.
    \]
    For $t \in [0,1]$ and $|n| > L(q)$, we have $1/(t + n) = 1/n + O(1/n^2)$,
    so again from the triangle inequality we have
    \begin{multline}\label{eq:supbound}
        \sup_{t\in[0,1]} \abs{F_{k,\chi}(t) - \tilde F_{k,\chi}(t)}^2 \ll
        \\
            \abs{
            \sum_{L(q) < |n| < q}
            \frac{e((\alpha + \beta)n)}
            {n}
            \chibar(k - n)
        }^2 +
            \abs{
            \sum_{L(q) < |n| < q}
            \frac{e(\alpha n)}
            {n}
            \chibar(k - n)
        }^2
             + \frac{1}{L(q)^{2}}.
     \end{multline}
    Now taking the average over all nontrivial characters we find that
    \begin{multline*}
        \frac{1}{q-2}\sum_{\chi \ne \chi_0} \sup_{t\in[0,1]} \abs{F_{k,\chi}(t) - \tilde F_{k,\chi}(t)}^2 \ll \\
        \frac{1}{q-2}
        \sum_{\chi \mod q}
            \abs{
            \sum_{L(q) < |n| < q}
            \frac{e((\alpha + \beta)n)}
            {n}
            \chibar(k - n)
        }^2 + \\
        \frac{1}{q-2}
        \sum_{\chi \mod q}
            \abs{
            \sum_{L(q) < |n| < q}
            \frac{e(\alpha n)}
            {n}
            \chibar(k - n)
        }^2
        + \frac{1}{L(q)^{2}}
        \\
        \ll 
        \frac{q-1}{q-2}\sum_{|n| > L(q)} \frac{1}{n^2} + \frac{1}{L(q)^2}
        \ll
        \frac{1}{L(q)}.
    \end{multline*}
    Upon applying Markov's inequality it follows that the number of primitive $\chi$ such that
    $\|F_{k, \chi,\alpha,\beta} - \tilde F_{k, \chi,\alpha,\beta}\| > z$ is bounded by
    $O\left(\frac{q}{z^2 L(q)}\right)$.

    We now suppose that $h:\mathscr C[0,1] \rightarrow \R$ is a bounded Lipschitz function with \linebreak
    $\abs{h(x_1) - h(x_2)} \le K\|x_1 - x_2\|$ and $|h(x)| \le B$ for all $x_1$ and
    $x_2$. Then for any $z$ we have
    \[
    \begin{split}
        \E h(F_{k,q,\alpha,\beta}) - \E h(\tilde F_{k,q,\alpha,\beta}) &= 
        \E \left[h(F_{k, q,\alpha,\beta}) - h(\tilde F_{k, q,\alpha,\beta})\right] \\
        &\le
        Kz + 2B\prob\left[\Cnorm{F_{k, \chi,\alpha,\beta} - \tilde F_{k, \chi,\alpha,\beta}} > z\right] \\
        &\le Kz + O\left(\frac{B}{z^2 L(q)} + \frac{B}{q}\right)
    \end{split}
    \]
    (the term $B/q$ arises from the contribution of the trivial character).
    Choosing $z = L(q)^{-1/3}$ yields the desired conclusion.

    For the quadratic case, we start with \eqref{eq:supbound} and average over
    $k$.
    Notice that
    \begin{multline*}
        \frac{1}{q}
        \sum_{k=0}^{q-1}
            \abs{
            \sum_{L(q) < |n| < q}
            \frac{e(xn)}
            {n}
            \chibar(k - n)
        }^2
        \\
        \ll
        \frac{1}{q}
        \sum_{k=0}^{q-1}
            \abs{
            \sum_{L(q) < n < q}
            \frac{e(xn)}
            {n}
            \chibar(k - n)
        }^2
        +
        \frac{1}{q}
        \sum_{k=0}^{q-1}
            \abs{
            \sum_{L(q) < -n < q}
            \frac{e(xn)}
            {n}
            \chibar(k - n)
        }^2
    \end{multline*}
    and
    \[
        \frac{1}{q}
        \sum_{k=0}^{q-1}
            \abs{
            \sum_{L(q) < \epsilon n < q}
            \frac{e(xn)}
            {n}
            \chibar(k - n)
        }^2
        =
        \frac{1}{q} \hspace{.3em}
        \sum\sum_{\hspace{-2.2em}\substack{L(q) < \epsilon m < q \\ L(q) < \epsilon n < q } }
            \frac{e(xn - xm)}
            {nm}
            \sum_{k=0}^{q-1}
            \chibar(k - n)
            \chi(k - m)
    \]
    for each choice of $\epsilon = \pm 1$.
    When $\chi$ is the quadratic character, the inner sum is $-1$ unless $n = m$,
    in which case it is $q-1$. Thus
    \[
    \begin{split}
        \frac{1}{q}
        \sum_{k=0}^{q-1}
            \abs{
            \sum_{L(q) < \epsilon n < q}
            \frac{e(xn)}
            {n}
            \chibar(k - n)
        }^2
        &=
        \frac{q-1}{q} \sum_{L(q) < \epsilon n < q} \frac{1}{n^2} -
            \frac{1}{q} \sum_{\substack{L(q) < \epsilon n < q \\ L(q) < \epsilon m < q \\ n \ne m}}
            \frac{e(xn - xm)}{nm} \\
        &\ll \frac{1}{L(q)} + \frac{(\log q)^2}{q} \ll \frac{1}{L(q)}.
        \end{split}
    \]
    We therefore have exactly the same quality bounds when we average over $k$ as when we average
    over $\chi$ and the proof is now complete.
\end{proof}

\section{Tightness and continuity}
\begin{lemma}\label{lemma:tightness}
    For any fixed $\alpha$ and $\beta$,
    there is a constant $C$ such that for all $k$ and $q$ and
    for all $s,t \in [0,1]$, we have
\[
    \E \abs{\tilde F_{k,q,\alpha,\beta}(s) - \tilde F_{k,q,\alpha,\beta}(t)}^2
        \le C \abs{s-t}^2
\]
and similarly
\[
    \E \abs{\tilde G_{q,\alpha,\beta}(s) - \tilde G_{q,\alpha,\beta}(t)}^2
        \le C \abs{s-t}^2.
\]
\end{lemma}

\begin{proof}
    Write $c(n, t) = e(\alpha( n+t))(e(\beta(n + t)) - 1)$.
    From the definition of $\tilde F_{k,q,\alpha,\beta}$ we have
\begin{align*}
    \E \Big|\tilde F_{k,q,\alpha,\beta}(s)& -  \tilde F_{k,q,\alpha,\beta}(t)\Big|^2 \\
   & = \frac{1}{q-1} \sum_{\chi \mod q} 
    \abs{
        \frac{\tau(\chi)}{2\pi q^{1/2}}
        \sum_{\abs{n} \leq L(q)}
    \frac{c(n,s)\chibar(k - n)}{s + n} - 
    \frac{c(n,t)\chibar(k - n)}{t + n}
    }^2 \\
   & \leq \frac{1}{4\pi^2(q-1)} \sum_{\chi \mod q}
    \abs{
        \sum_{\abs{n} \leq L(q)}
        \chibar(k - n)\left(
        \frac{c(n,s)}{s + n} - 
        \frac{c(n,t)}{t + n}
        \right)
    }^2,
\end{align*}
where we have an inequality as the contribution from the principal character is smaller by a factor $1/q$. 
It is now convenient to take just half the sum, so take 
$\epsilon \in \{\pm 1\}$ which maximizes the sum
\[
\sum_{\chi \mod q}
    \abs{
        \sum_{0 \le \epsilon n \leq L(q)}
        \chibar(k -n)\left(
        \frac{c(n,s)}{s + n} - 
        \frac{c(n,t)}{t + n}
        \right)
    }^2,
\]
so that
\begin{align*}
    \E \Big|\tilde F_{k,q,\alpha,\beta}(s)& -  \tilde F_{k,q,\alpha,\beta}(t)\Big|^2 \\
   & \ll \frac{1}{q} \sum_{\chi \mod q}
    \abs{
        \sum_{0 \le \epsilon n \leq L(q)}
        \chibar(k - n)\left(
        \frac{c(n,s)}{s + n} - 
        \frac{c(n,t)}{t + n}
        \right)
    }^2.
\end{align*}
Then opening up the sum and using orthogonality we have
\[
    \E \abs{\tilde F_{k,q,\alpha,\beta}(s) - \tilde F_{k,q,\alpha,\beta}(t)}^2
    \ll
        \sum_{0 \le \epsilon n \leq L(q)}
        \abs{
        \frac{c(n,s)}{s + n} - 
        \frac{c(n,t)}{t + n}
    }^2
\]
as long as $L(q) < q$.

By exactly the same computations and again using the fact that
$\sum_{k=0}^{q-1} \chi_q(k)\chi_q(k+a) = q-1$ when $a \equiv 0$
and $-1$ otherwise, we find that
\begin{multline*}
    \E \abs{\tilde G_{q,\alpha,\beta}(s) - \tilde G_{q,\alpha,\beta}(t)}^2
    \ll 
        \sum_{0 \le \epsilon n \leq L(q)}
        \abs{
        \frac{c(n,s)}{s + n} - 
        \frac{c(n,t)}{t + n}
    }^2 \\
    + \frac{1}{q} 
    \sum_{\substack{0 \le \epsilon n \leq L(q) \\ 0 \le \epsilon m \leq L(q)} }
    \left(\frac{c(n,s)}{s+n} - \frac{c(n,t)}{t + n}\right)
    \left(\frac{\overline{c(m,s)}}{s+m} - \frac{\overline{c(m,t)}}{t + m}\right).
\end{multline*}

Now, the derivative of $c(n,t)/(t+n)$ is $O(1/(|n|+1))$ for $t \in [0,1]$, so
\[
    \abs{
        \frac{c(n,s)}{s + n} - 
        \frac{c(n,t)}{t + n}
    } \ll \frac{\abs{s-t}}{|n|+1}
\]
by the mean value theorem
and hence
\begin{multline*}
    \E \abs{ A(s) - A(t)}^2
    \ll
        \abs{s-t}^2
        \sum_{0 \le n \leq L(q)}
            \frac{1}{(n+1)^2}
        +
        \frac{\abs{s-t}^2}{q}\left(\sum_{0 \le n \leq L(q)} \frac{1}{n+1}\right)^2
        \\
    \ll
        \abs{s-t}^2,
\end{multline*}
where $A = \tilde F_{k,q,\alpha,\beta}$ or $\tilde G_{q, \alpha, \beta}$.
\end{proof}
We will now show that all our processes are almost surely continuous.
\begin{lemma}
    For any real $\alpha$ and $\beta$ and any integer $k$,
    each of the random processes $F_{k,\alpha,\beta}(t)$,
    $F_{\alpha,\beta}(t)$ and $G_{\alpha,\beta}(t)$ is almost surely
    continuous. 
\end{lemma}
\begin{proof}
The calculations we give here for $F_{k,\alpha,\beta}$ work equally well
for $F_{\alpha,\beta}$ and $G_{\alpha,\beta}$ with simple modifications (in the latter case, those are simpler as the random variables involved are independent). We make a change of variables and write
\[
    F_{k, \alpha,\beta}(t) = \bX \frac{e(\alpha t)}{2\pi i}
    \sum_{l \in \Z}
    \frac{e(\alpha (k-l))(e(\beta(k - l + t)) - 1)}
    {k - l + t}
    \bW(l).
\]
To prove continuity, we can ignore the multipliers out front and any finite number of terms, so we stay away from the region where the
denominator is $0.$ Consider the sum
\[
    S = 
    \sum_{\substack{l \in \Z \\ |l| > 2\abs{k}+2}}
    \frac{e(\alpha (k-l)) (e(\beta(k - l + t)) - 1)}
    {k - l + t}
    \bW(l).
\]
and note
\[
    \frac{1}{k - l + t} = \frac{1}{-l} + \frac{k+t}{l(k - l + t)}.
\]
Consequently
\begin{multline*}
    S = 
    e(\alpha k)
    \sum_{\substack{l \in \Z \\ |l| > 2\abs{k}+2}}
    \frac{e(-\alpha l)}{l}\bW(l)
    -
    e((\alpha + \beta)k + \beta t)
    \sum_{\substack{l \in \Z \\ |l| > 2\abs k + 2}}
    \frac{e(-(\alpha + \beta)l)}{l} \bW(l)
    \\
    +
    \sum_{\substack{l \in \Z \\ |l| > 2\abs k + 2}}
    \frac{e(\alpha (k-l)) (e(\beta(k - l + t)) - 1)(k+t)}
    {l(k - l + t)}
    \bW(l).
\end{multline*}
The last sum is uniformly absolutely convergent on $[0, 1]$, hence continuous for
any realization of $\bW(l)$. So the only problems with convergence
may arise from the first two sums, which are independent of $t$. But moment bounds easily show that
they converge with probability 1 (alternatively, we may apply the Rademacher--Menshov orthogonal series theorem). More generally, we will shortly compute
moment bounds for all finite-dimensional distributions of
$F_{k,\alpha,\beta}$ which imply almost sure convergence for every $t$.
\end{proof}

\section{Moment computations}\label{sec:moments}
We now examine the mixed moments of the finite-dimensional distributions
of our various random processes.
\subsection{Varying character \texorpdfstring{$\chi \mod q$}{}}
We use the following bound for the sum of characters weighted by the Gauss sum.
\begin{lemma}\label{lemma:gauss-sum}
    For any $a$ mod $q$ and $n \ge 1$ we have
    \[
        \abs{\frac{1}{q-1}\sum_{\chi \mod q} \chi(a) \tau(\chi)^n} \le n q^{(n-1)/2}. 
    \]
\end{lemma}
\begin{proof}
    This is essentially the same as Lemma 8.3 of \cite{GS-large-character-sums},
    except that we sum over all characters instead of just the odd characters.
    Putting in the definition of the Gauss sum and using the orthogonality
    of Dirichlet characters, this becomes a hyper-Kloosterman sum, and bounds
    of Deligne apply.
\end{proof}

\begin{proposition}\label{prop:chi-moments}
    For $t_1, t_2, \ldots, t_J \in [0,1]$, $r_1, \ldots, r_J, s_1, \ldots, s_J \in \N$,
    fixed $\alpha$ and $\beta$, and for fixed $k$, we have
\[
    \E \left[ \prod_{j=1}^J \tilde F_{k,q,\alpha,\beta}(t_j)^{r_j} \overline{\tilde F_{k,q,\alpha,\beta}(t_j)^{s_j}} \right]
    \longrightarrow
    \E \left[ \prod_{j=1}^J F_{k,\alpha,\beta}(t_j)^{r_j} \overline{F_{k,\alpha,\beta}(t_j)^{s_j}} \right]
\]
as $q \rightarrow \infty$ over the primes. Moreover, this moment sequence of
$F_{k,\alpha,\beta}$ is determinate, so the finite-dimensional distributions of
$\tilde F_{k,q,\alpha,\beta}$ tend to those of $F_{k,\alpha,\beta}$ as $q$ tends to infinity
over the primes.
\end{proposition}
\begin{proof}
    Let $r = \sum r_j$ and $s = \sum s_j$ and write 
\[
    \mathbf{m}_j = (m_{j,1}, \dots, m_{j, r_j}, m_{j, r_j + 1}, \dots, m_{j, r_j+s_j})
\]
and
\[
a(\bm_j; t)
=
    \prod_{i=1}^{r_j} c(m_{j,i}, t)
    \prod_{i=r_j+1}^{r_j + s_j} c(-m_{j,i}, -t)
\]
where
\[
    c(n,t) = \frac{e(\alpha n)(e(\beta(n + t)) - 1)}{t+n}.
\]
Using the above notation, we have

\begin{multline}\label{eq:moment-mess}
    \prod_{j=1}^J \tilde F_{k,\chi,\alpha,\beta}(t_j)^{r_j} \overline{\tilde F_{k,\chi,\alpha,\beta}(t_j)}^{s_j}
    \\
    =
    e(\alpha T)
    \frac{\tau(\chi)^r \overline{\tau(\chi)}^s}{(2\pi i)^{r+s}q^{(r + s)/2}}
 \sum_{\bm_1, \dots, \bm_J}
 \prod_{j=1}^J\prod_{\ell=1}^{r_j}\chibar(k - m_{j,\ell})
 \prod_{\ell=r_j + 1}^{r_j + s_j}\chi(k - m_{j,\ell})
 \prod_{j=1}^J a(\bm_j; t_j),
\end{multline}
where
\[
    T = \sum_{j=1}^J (r_j - s_j)t_j
\]
and the quantity
$$ \bm_j= (m_{j,1}, \dots, m_{j, r_j}, m_{j, r_j+1}, \dots, m_{j, r_j+s_j})$$
ranges over all $(r_j+s_j)$-tuples of integers $|m_{j, \ell}|\leq L(q)$.
Note that
\[
    \frac{\tau(\chi)^r \overline{\tau(\chi)^s}}{q^{(r+s)/2}} = \frac{\tau(\chi)^{r-s}}{q^{(r-s)/2}}
\]
for any nontrivial character mod $q$, so bringing the sum over characters
inside, we have
\begin{multline*}
    (2\pi i)^{r+s}\E \left[ \prod_{j=1}^J \tilde F_{k,q,\alpha,\beta}(t_j)^{r_j} \overline{\tilde F_{k,q,\alpha,\beta}(t_j)}^{s_j} \right] = \\
 \frac{e(\alpha T)}{q - 1}
 \sum_{\bm_1, \dots, \bm_J}
 \prod_{j=1}^J a(\bm_j; t_j)
 \sum_{\chi \mod q}
 \frac{\tau(\chi)^{r-s}}{q^{(r - s)/2}}
 \prod_{j=1}^J\prod_{\ell=1}^{r_j}\chibar(k - m_{j,\ell})
 \prod_{\ell=r_j+1}^{r_j + s_j}\chi(k - m_{j,\ell}) \\
 + \frac{e(\alpha T)}{q-1}
 \sum_{\bm_1, \dots, \bm_J}
 \prod_{j=1}^J a(\bm_j; t_j)
 \left( \frac{(-1)^{r+s}}{q^{(r+s)/2}} - \frac{(-1)^{r-s}}{q^{(r-s)/2}} \right)
 \prod_{j=1}^J\prod_{\ell=1}^{r_j}\chibar_0(k - m_{j,\ell})
 \prod_{\ell=r_j+1}^{r_j + s_j}\chi_0(k - m_{j,\ell}).
\end{multline*}
For $r > s$, (the case $r < s$ works identically, swapping $\tau(\chi)$ with
$\overline{\tau(\chi)}$) we can now insert the bound from Lemma
\ref{lemma:gauss-sum} for the inner sum to find that
\[
    \E \left[ \prod_{j=1}^J \tilde F_{k,q,\alpha,\beta}(t_j)^{r_j} \overline{\tilde F_{k,q,\alpha,\beta}(t_j)}^{s_j} \right]
     \ll (r-s)q^{-1/2}
 \sum_{\bm_1, \dots, \bm_J}
 \prod_{j=1}^J \abs{a(\bm_j; t_j)} \ll \frac{(\log L(q))^{r+s}}{q^{1/2}},
\]
where for the final inequality we notice that
\[
    \sum_{\bm_1, \dots, \bm_J}
    \prod_{j=1}^J \abs{a(\bm_j; t_j)} =
    \prod_{j=1}^J
    \left(\sum_{|m| \leq L(q)} |c(m, t_j)|\right)^{r_j + s_j}
    \ll (\log L(q))^{r+s}.
\]
For the limiting process, the off diagonal case is much easier: rotation
invariance tells us that
\[
\E \left[ \prod_{j=1}^J F_{k,\alpha,\beta}(t_j)^{r_j} \overline{F_{k,\alpha,\beta}(t_j)}^{s_j} \right]
 = 0
\]
whenever $r \ne s$.

Meanwhile, when $r=s$, the terms $\tau(\chi)^{|r-s|}/q^{|r-s|/2}$ go away and
\[
 \frac{1}{q-1}
 \sum_{\chi \mod q}
 \prod_{j=1}^J\prod_{\ell=1}^{r_j}\chibar(k - m_{j,\ell})
 \prod_{\ell=r_j + 1}^{r_j+s_j}\chi(k - m_{j,\ell})
 = \cS_q(k, \bm),
\]
say, where
\[
 \cS_q(k, \bm) =
 \left\{
 \begin{array}{ll}
     1 & \textrm{ if } 
     \prod_{j=1}^J\prod_{\ell=1}^{r_j}(k - m_{j, \ell}) \equiv
     \prod_{j=1}^J\prod_{\ell=r_j + 1}^{r_j + s_j}(k - m_{j, \ell})  \not \equiv 0 \mod q\\
     0 & \textrm{ otherwise.}
 \end{array}
 \right.
\]
Then
\begin{multline*}
\E \left[ \prod_{j=1}^J \tilde F_{k,q,\alpha,\beta}(t_j)^{r_j} \overline{\tilde F_{k,q,\alpha,\beta}(t_j)}^{s_j} \right] =
 (2\pi i)^{-(r+s)}e(\alpha T)
 \sum_{\bm_1, \dots, \bm_J}\cS_q(k, \bm)
 \prod_{j=1}^J a(\bm_j; t_j)
 \\ + O\left(\frac{(\log L(q))^{2r}}{q}\right)
 ,
\end{multline*}
with the error term accounting for the contribution of the trivial character.
When $r=s$, the same computations yield
\[
\E \left[ \prod_{j=1}^J F_{k,\alpha,\beta}(t_j)^{r_j} \overline{F_{k,\alpha,\beta}(t_j)}^{s_j} \right] = \\
    (2\pi i)^{-(r+s)}
    e(\alpha T)
 \sum_{\bm_1, \dots, \bm_J} \cS(k, \bm)
 \prod_{j=1}^J a(\bm_j; t_j)
,
\]
where there is now no restriction on the size of $\bm_j$ and
\[
 \cS(k, \bm) =
 \left\{
 \begin{array}{ll}
     1 & \textrm{ if } 
     \prod_{j=1}^J\prod_{\ell=1}^{r_j}(k - m_{j, \ell}) =
     \prod_{j=1}^J\prod_{\ell=r_j + 1}^{r_j + s_j}(k - m_{j, \ell}) \ne 0 \\
     0 & \textrm{ otherwise.}
 \end{array}
 \right.
\]
As $q \rightarrow \infty$, for any fixed $r, s$ and $k$, it is clear that
\[
 \sum_{
     \substack{\bm_1, \dots, \bm_J \\ \abs{m_{j,l}} \leq L(q) }
    } \cS_q(k, \bm)
 \prod_{j=1}^J a(\bm_j; t_j)
 \longrightarrow
 \sum_{\bm_1, \dots, \bm_J}\cS(k, \bm)
 \prod_{j=1}^J a(\bm_j; t_j)
 ,
\]
provided that the right hand side is absolutely convergent, as the congruence
conditions in the definition of $\cS_q$ play no role for large enough $q$.
Now note that $c(n,t) \ll 1/(\abs{n} + 1)$, so $a(\bm_j, t) \ll
\prod_{i=1}^{r_j + s_j} (\abs{m_{j,i}} + 1)^{-1}$.
To see that the right hand side is absolutely convergent it is convenient to
reorder by the size of $\prod_{j=1}^{J}\prod_{i=1}^{r_j} (k - m_{j,i})$. Then
we see that
\begin{equation}\label{eg:moment-divisor-bound}
\E \left[ \prod_{j=1}^J F_{k,\alpha,\beta}(t_j)^{r_j} \overline{F_{k,\alpha,\beta}(t_j)}^{s_j} \right]
    \ll \sum_{x = 1}^\infty \frac{1}{x^2}
    \sum_{\substack{n_1, \ldots, n_r \in \Z \\ m_1, \ldots, m_r \in \Z
        \\ (n_1 + k)\cdots(n_r + k) = x \\ (m_1 + k)\cdots(m_r + k) = x }}
        1
        \ll 2^{2r} \sum_{x=1}^\infty \frac{d_r(x)^2}{x^2}.
\end{equation}
This bound on the moments is also sufficient to prove that the moments
determine the distribution. From \cite{BG-maximum}*{Proposition 3.2},
for example, we have
\[
    \sum_{x=1}^\infty \frac{d_r(x)^2}{x^2} \le C^{2r}(\log 2r)^{2r}
\]
for some constant $C$. So for any collection of $t_j$ we see that
\begin{multline*}
    \sum_{r=1}^\infty \abs{\E \left[ \prod_{j=1}^J F_{k,\alpha,\beta}(t_j)^{r_j} \overline{F_{k,\alpha,\beta}(t_j)}^{s_j} \right]}^{-1/2r}
    \gg \\
    \sum_{r=1}^\infty \left(\frac{1}{(2C)^{2r} (\log (2r) )^{2r}}\right)^{1/2r}
    = \sum_{r=1}^\infty \frac{1}{2C \log(2r)} = \infty.
\end{multline*}
Hence the moment sequences satisfy the multivariate complex
Carleman condition (see \cite{moment-problem}*{Theorem 15.11}) and are determinate.
\end{proof}

We now turn to the remaining parts of Theorem \ref{thm:random-process}, i.e.
we study what happens when $k$ varies with $q$. In our bounds for the 
off-diagonal moments, the size of $k$ plays no role. On the other hand, the
$\bm$ which contribute to positive values of $\cS(k, \bm)$ depend on some
shifted factorizations and $k$ visibly affects the moments.

When $k$ grows with $q$, it will generally
be the case that $\cS_q(k, \bm)$ is nonzero exactly
when the second half of $\bm$ is a permutation of the first half, and
thus the influence of the multiplicativity of $\chi$ goes away.
However, if $k$ is growing fast enough there will be some primes
where the multiplicativity is important, so we need to remove these primes
to get convergence of the moments. Specifically, we will prove the following.

\begin{proposition} \label{prop:chi-moments2}
    Fix real numbers $\alpha$ and $\beta$. There exist functions $L(q)\to\infty$ such that: 
\begin{itemize}
    \item[(1)] For any sequence of positive integers $k_q\to\infty$ such that
        $k_q = q^{o(1)}$, the finite-dimensional distributions of $\tilde
        F_{k_q, q, \alpha, \beta}(t)$ converge to those of $F_{\alpha,
        \beta}(t)$ as $q\to\infty$ over the primes.
    \item[(2)] For any sequence of positive integers $k_q\to\infty$ such that
        $k_q = o(\pi(q))$, there exists a full density subset of the primes $\cP$
        such that the finite-dimensional distributions of $\tilde F_{k_q, q,
        \alpha, \beta}(t)$ converge to those of $F_{\alpha, \beta}(t)$ as $q\in\cP$ tends to infinity. 
    \item[(3)] For any irrational $\theta$, there exists a full density subset
        of the primes $\cP_\theta$ such that the finite-dimensional distributions
        of $\tilde F_{\floor{\theta q}, q, \alpha, \beta}(t)$ converge to those
        of $F_{\alpha,\beta}(t)$ as $q \in \cP_\theta$ tends to infinity.
\end{itemize}
\end{proposition}
\begin{proof}
    We again match the moments of the random processes.
    We will abbreviate $F_{\alpha,\beta} = F$ and $\tilde F_{k_q, q,
    \alpha,\beta} = \tilde F_{k,q}$, and again write $r = \sum r_j$ and $s =
    \sum s_j$.
    We first examine $F$. Recall that
    \[
        F(t) = \frac{e(\alpha t)}{2\pi i}
        \sum_{l \in \Z}
        \frac{e(\beta(l + t)) - 1}
        {l + t}
        \bX(l).
    \]
    $F$ is clearly unchanged in distribution if we multiply by another independent variable
    $\bX$ uniformly distributed on the unit circle. So when $r \ne s$ the
    moments vanish:
    \[
        \bE \prod_{j=1}^J F(t_j)^{r_j} \overline{F(t_j)}^{s_j} = 
        \bE \prod_{j=1}^J \bX^{r_j} F(t_j)^{r_j} \overline{\bX^{s_j} F(t_j)^{s_j}}
        = \bE \bX^{r-s} 
        \bE \prod_{j=1}^J F(t_j)^{r_j} \overline{F(t_j)}^{s_j} = 0.
    \]
    When $r=s$ we can compute as in the previous proof.
    We find that
\[
\E \left[ \prod_{j=1}^J F(t_j)^{r_j} \overline{F(t_j)}^{s_j} \right] = (2\pi i)^{-(r+s)} 
    e(\alpha T)
 \sum_{\bm_1, \dots, \bm_J}\cS(\bm)
 \prod_{j=1}^J a(\bm_j; t_j)
 ,
\]
where again there is no restriction on the size of $\bm_j$ and
this time
\[
 \cS(\bm) =
 \left\{
 \begin{array}{ll}
     1 & \textrm{ if } 
     m_{1, 1}, \ldots, m_{J,r_J} \textrm{ is a permutation of } m_{1, r_1 + 1}, \ldots m_{J,r_J + s_J} \\
     0 & \textrm{ otherwise.}
 \end{array}
 \right.
\]
The absolute convergence of this series and upper bounds for its
moments follow as before, again showing that this process is determined
by its moments.
Now we go back and consider the moments of $\tilde F_{k, q}$.
When $r \ne s$, the upper bounds from the previous proof work just as well,
as they did not involve $k$.

When $r = s$ we are concerned with whether
$\cS_q(k_q, \bm) = \cS(\bm)$ for all large $q$ and small $\bm$.
Equivalently, we want to know that the only solutions to
\[
    (k_q+n_1)(k_q+n_2)\cdots(k_q+n_r) \equiv (k_q+m_1)(k_q+m_2)\cdots(k_q+m_r) \pmod q
\]
with $|n_l|, |m_l| \leq L(q)$ are those where $m_1, \ldots, m_r$ is
a permutation of $n_1, \ldots, n_r$. In general, this may not be the case. For
each $n_1, \ldots, n_r$ and $m_1, \ldots, m_r$ we consider the polynomial
\[
    h_{\bn, \bm}(x) = (x+n_1)(x+n_2)\cdots(x+n_r) - (x+m_1)(x+m_2)\cdots(x+m_r).
\]
If $\bn$ is not a permutation of $\bm$, this is not the zero
polynomial. The coefficients are bounded by some power of $L(q)$,
and it is exactly the content of Lemma \ref{lemma:goodprimes} (below in Section \ref{primeRemSect}) that
there exist choices of $L(q)$ such that: in case
(1), there are only finitely many primes such that $h(k_q)\equiv 0\pmod q$ for
any non-zero polynomial with small coefficients; in each of cases (2) and (3),
there is a full
density subset of the primes $\cP$ such that $h(k_q)\not\equiv 0\pmod q$ for any
non-zero polynomial with small coefficients when $q\in\cP$ is large enough.
\end{proof}

\subsection{Quadratic characters}

We now turn to the quadratic character case and finish the proof of Theorem \ref{thm:random-quadratic-process}. Recall that
$G_{q, \alpha, \beta}$ is the random process given by
choosing $k$ mod $q$ uniformly at random and forming the
function
\[
    \frac{q^{1/2}}{\tau(\chi_q)}
    F_{k, \chi_q, \alpha, \beta}(t) = \frac{1}{\tau(\chi_q)}
        \sum_{\alpha q < n \leq (\alpha + \beta)q} \chi_q(n)e(n(k + t)/q),
\]
where $\chi_q$ is the quadratic character mod $q$.
However, rather than working with this process directly, we work with
the truncated approximation $\tilde G_{q, \alpha, \beta}$ given by choosing
$k$ at random and forming the function
\[
    \frac{q^{1/2}}{\tau(\chi_q)}
    \tilde F_{k,\chi_q,\alpha,\beta}(t)
    = \frac{e(\alpha t)}{2\pi i}
    \sum_{\abs{l} \leq L(q)}
    \frac{e(\alpha l) (e(\beta(l + t)) - 1)}
    {l + t }
    \chi_q(k-l).
\]

\begin{proposition}\label{prop:quadratic-moments}
    For $t_1, t_2, \ldots, t_J \in [0,1]$, $r_1, \ldots, r_J, s_1, \ldots, s_J \in \N$,
    fixed $\alpha$ and $\beta$, we have
\[
    \E \left[ \prod_{j=1}^J \tilde G_{q,\alpha,\beta}(t_j)^{r_j} \overline{\tilde G_{q,\alpha,\beta}(t_j)^{s_j}} \right]
    \longrightarrow
    \E \left[ \prod_{j=1}^J G_{\alpha,\beta}(t_j)^{r_j} \overline{G_{\alpha,\beta}(t_j)^{s_j}} \right]
\]
as $q \rightarrow \infty$ over the primes. Moreover, this moment sequence of
$G_{\alpha,\beta}$ is determinate, so the finite-dimensional distributions of
$\tilde G_{q,\alpha,\beta}$ tend to those of $G_{\alpha,\beta}$ as $q$ tends to infinity
over the primes.
\end{proposition}

\begin{proof}
    We begin with the same notation and combinatorics as in the proof of
    Proposition \ref{prop:chi-moments}, recalling 
    \eqref{eq:moment-mess}:
\begin{multline*}
   \frac{q^{(r+s)/2}}{\tau(\chi_q)^r\overline{\tau(\chi_q)}^s}\prod_{j=1}^J  \tilde F_{k,\chi_q,\alpha,\beta}(t_j)^{r_j} \overline{\tilde F_{k,\chi_q,\alpha,\beta}(t_j)}^{s_j}
    \\
    =
    e(\alpha T)
    \frac{1}{(2\pi i)^{r+s}}
 \sum_{\bm_1, \dots, \bm_J}
 \prod_{j=1}^J\prod_{\ell=1}^{r_j}\chi_q(k - m_{j,\ell})
 \prod_{\ell=r_j + 1}^{r_j + s_j}\chi_q(k - m_{j,\ell})
 \prod_{j=1}^J a(\bm_j; t_j).
\end{multline*}
We take the sum over $k$ to compute the expectation
\begin{multline*}
\E \left[ \prod_{j=1}^J \tilde G_{q,\alpha,\beta}(t_j)^{r_j} \overline{\tilde G_{q,\alpha,\beta}(t_j)}^{s_j} \right] = \\
\frac{e(\alpha T)}{q(2\pi i)^{r+s}}
 \sum_{\bm_1, \dots, \bm_J}
 \prod_{j=1}^J a(\bm_j; t_j)
 \sum_{k \mod q}
 \prod_{j=1}^J\prod_{\ell=1}^{r_j}\chi_q(k - m_{j,\ell})
 \prod_{\ell=r_j + 1}^{r_j + s_j}\chi_q(k - m_{j,\ell}).
\end{multline*}
We use the Weil bound for character sums (see \cite{IwKo}*{Corollary 11.24},
for example) in the form
\[
 \frac{1}{q}
 \sum_{k \mod q}
 \prod_{j=1}^J\prod_{\ell=1}^{r_j}\chi_q(k - m_{j,\ell})
 \prod_{\ell=r_j + 1}^{r_j + s_j}\chi_q(k - m_{j,\ell})
 =
\bE\left(\prod_{j=1}^J\prod_{\ell=1}^{r_j+s_j} \bY(m_{j, \ell})\right) +
O\left(\frac{r+s}{\sqrt{q}}\right).
\]

Thus, we derive
\begin{align*}\
    \E \Bigg[ \prod_{j=1}^J & \tilde G_{q,\alpha,\beta}(t_j)^{r_j} \overline{\tilde G_{q,\alpha,\beta}(t_j)}^{s_j} \Bigg] \\
        & = \frac{e(\alpha T)}{(2\pi i)^{r+s}}
    \sum_{\mathbf{m_1}, \dots, \mathbf{m_J}}
        \prod_{j=1}^J a(\mathbf{m_j}; t_j)
        \bE\left(\prod_{j=1}^J\prod_{\ell=1}^{r_j+s_j} \bY(m_{j, \ell})\right)  + E_1,
\end{align*}
where 
\begin{align*}
 E_1 &
    \ll_{r, s} \frac{1}{\sqrt{q}} \sum_{\mathbf{m_1}, \dots, \mathbf{m_J}}
        \prod_{j=1}^J |a(\mathbf{m_j}; t_j)|
        = \frac{1}{\sqrt{q}}\prod_{j=1}^J
            \left(\sum_{|m|\leq L(q)} |c(m ,t_j)|\right)^{r_j+s_j}
            \hspace{-1.5em}\ll_{r, s} \frac{(\log L(q))^{r+s}}{\sqrt{q}},
 \end{align*}
since uniformly for $t\in [0,1]$ we have $|c(m ,t)| \ll 1/|m|$ for $|m|\geq 2$ and hence
\[
    \sum_{|m|\leq L(q)} |c(m ,t)| \ll
        1 + \sum_{2\leq |m|\leq L(q)} \frac{1}{|m|} \ll \log L(q).
\]
Formally, we get the same expression for the moments of $G_{\alpha,\beta}:$
\[
\bE\left(\prod_{j=1}^J
{G}_{\alpha,\beta}(t_j)^{r_j} \overline{{G}_{\alpha,\beta}(t_j)}^{s_j}\right)
=
    \frac{e(\alpha T)}{(2\pi i )^{r+s}}
    \sum_{\mathbf{m_1}, \dots, \mathbf{m_J}}
        \prod_{j=1}^J a(\mathbf{m_j}; t_j)
        \bE\left(\prod_{j=1}^J\prod_{\ell=1}^{r_j+s_j} \bY(m_{j, \ell})\right),
\]
where there is now no restriction on the size of $m_{j,l}$. So it is clear
that the moments of $\tilde G_{q,\alpha,\beta}$ converge to those of
$G_{\alpha,\beta}$ provided that this sum is absolutely convergent.
The expectation
\[
    \bE\left(\prod_{j=1}^J\prod_{\ell=1}^{r_j+s_j} \bY(m_{j, \ell})\right)
\]
is $0$ unless the terms $m_{j,l}$ coincide in pairs, in which case it is $1$.
If the $m_{j,l}$ coincide in pairs, then in particular their product is a
square. Similar to the bound \eqref{eg:moment-divisor-bound} in the proof of
Proposition \ref{prop:chi-moments} we have the upper bound
\begin{align*}
    \sum_{\mathbf{m_1}, \dots, \mathbf{m_J}}
        \abs{\prod_{j=1}^J a(\mathbf{m_j}; t_j)
        \bE\left(\prod_{j=1}^J\prod_{\ell=1}^{r_j+s_j} \bY(m_{j, \ell})\right)}
    & 
    \ll
    2^{r+s} \sum_{x=1}^\infty \frac{d_{r+s}(x^2)}{x^2} \\
    &
    \ll 2^{r+s} \sum_{x=1}^\infty \frac{d_{r+s}(x)^2}{x^2} \\
    &
    \ll C^{2(r+s)}(\log 2(r+s))^{2(r+s)}
\end{align*}
for some constant $C$. This similarly provides an upper bound for the moments
which proves that the moment sequence is determinate, as before.
\end{proof}

\section{Prime removal lemmas}\label{primeRemSect}
When computing the moments in parts (2) and (3) of Proposition
\ref{prop:chi-moments2} (and hence in parts (3) and (4) of Theorem \ref{thm:random-process}),
it is necessary to remove a density zero subset of the primes in order
to get convergence. Here we prove the lemmas that we use to do this.
\begin{lemma}\label{lemma:fpalpha}\phantom{blah}
   \begin{itemize} \item[(1)]  Suppose that 
    $k_p\to\infty$ is a sequence of integers indexed by the primes,
    with $\#\{k_p : p\leq X\} =: \psi(X)$ and $k_p\ll p$. There exists a function $E(X)\to\infty$ with $E(X)\leq \log X$ such that for all polynomials $f(n) = \sum_{j=0}^d e_j n^j \in
    \Z[n]$ of degree $d$ with $\abs{e_j}\leq E(X)$ and all large enough $X$, we have
    \[
        \#\left\{p\leq X: f(k_p) \equiv 0 \mod p \right\} 
            \ll (d+1)\psi(X)\log(2+\max\abs{e_j}).
    \]
    
    \item[(2)] Let $\theta$ be an irrational number. There exists a function $E_{\theta}(X)\to\infty$ with $E_{\theta}(X)\leq \log X$ such that for all polynomials $f(n) = \sum_{j=0}^d e_j n^j \in
    \Z[n]$ of degree $d$ with $\lvert e_j\rvert\leq E_{\theta}(X)$ and all large enough $X$, we have
    \[
        \#\left\{p\leq X : f(\floor{\theta p}) \equiv 0 \mod p \right\} \ll
        X^{1/2}\left(d + 1 + \frac{\log(2+E_{\theta}(X))}{\log X}\right).
    \]
    \end{itemize}
\end{lemma}
\begin{proof}
    For part (1) of the lemma, note that if $k_p=k$ is large enough and $f(k)\neq 0$, then $p\geq \sqrt{k}$ and $$k^{\frac{1}{2}\#\{p\leq X : k_p=k, p\mid f(k)\}}\ll \abs{f(k_p)} \ll \max{\abs{e_j}}k^d.$$ Then for such $k$ we have 
    $$\#\{p\leq X : k_p = k, p\mid f(k)\}\ll d + \log (2+\max{\abs{e_j}}).$$ Summing over the $\psi(X)$
    values for $k$ in the range $p\leq X$, and noting that $E(X)$ may be chosen to grow slowly enough to ensure that the number of primes $p\leq X$ with $f(k_p) = 0$ is also at most $\psi(X)$ (by Cauchy's bound on the roots of a polynomial), we obtain the desired conclusion. 

    For part (2), we may assume that $\theta\in(0,1)$ and we consider the
     number of primes between
    $\sqrt X$ and $X$ such that $f(\floor{\theta p}) \equiv 0 \pmod p$. Choose some $Q < X$
    (which will depend only on $X$ and diverges with $X$) and let $a, q$ be coprime such that
    $\abs{\theta - a/q} \le 1/(qQ)$ with  $q\leq Q$. Note that as $X\to\infty$ (and thus $Q\to\infty$), we also have that $q\to\infty$, since $\theta$ is irrational. As such, we may choose a function $E_{\theta}(X)\to\infty$ that guarantees $q>\lvert e_d\rvert$. 

    Consider some prime $p \in [\sqrt X, X]$.
    Then $\theta p - ap/q = \lambda p$ for some $\lambda$ with 
    \[
        |\lambda p| \le \frac{p}{qQ}.
    \]
    Write $p = kq + r$ for some $k$ and $r$ with $1 \le r < q$,
    where $r$ will be coprime to $q$ because $p$ is a prime larger
    than $q$. It follows that
    \[
        \floor{\theta p} = \floor{\frac{ap}{q} + \lambda p} = \floor{ak + \frac{ar}{q} + \lambda p}
             = ak + c + \varepsilon,
    \]
    say, where $c = \floor{ar/q}$ is completely determined by $r$
    (once $a$ and $q$ have been fixed) and
    $\abs{c} < q$ and $\abs{\varepsilon} \le 1 + p/qQ$.
    Now the condition that $f(\floor{\theta p}) \equiv 0 \mod p$ is that
    \[
        f(ak + c + \varepsilon) \equiv 0 \mod {kq + r}.
    \]
    Write $f(n) = \sum_{j=0}^d e_j n^j$, so this is
    \[
        \sum_{j=0}^d e_j (ak + c + \varepsilon)^j \equiv 0 \mod {kq + r}.
    \]
    We multiply by $q^d$ to get
    \[
        \sum_{j=0}^d e_j (qak + qc + q\varepsilon)^j q^{d-j} \equiv 0 \mod {kq + r}.
    \]
    We now notice that $qak \equiv -ar \mod{kq + r}$, removing the dependence
    on $k$, that is
    \[
        \sum_{j=0}^d e_j (-ar + qc + q\varepsilon)^j q^{d-j} \equiv 0 \mod {kq + r}.
    \]
    Next notice that $\sum_{j=0}^d e_j (-ar + qc + q\varepsilon)^j q^{d-j}$
    is just a number, which for a given $X$ depends only on $r$ and $\varepsilon$. It
    is bounded by $O(\max\abs{e_j}(q + X/Q)^{2d})$ because $a, r, $ and $c$ are $< q$
    and $\varepsilon \ll X/qQ$, and
    it is nonzero because it is $\equiv e_d(-ar)^d \mod q$ and $a$ and $r$ are
    coprime to $q$, whereas $q>\lvert e_d\rvert$. So it is divisible by at most
    \[
        O\left(\frac{2d \log (q + X/Q) + \log(2+\max\abs{e_j})}{\log X}\right)
    \]
    primes
    in the range $[\sqrt X, X]$. There are $q$ different possibilities for $r$
    and $ \ll \max(1, X/qQ)$ different possibilities for $\varepsilon$, so there 
    can be at most
    \[
        O\left(\frac{\max(q, X/Q)(2d\log (q + X/Q) + \log(2+\max\abs{e_j}))}{\log X} \right)
    \]
    primes between $\sqrt X$ and $X$ such that $f(\floor{\theta p}) \equiv 0 \mod p$. 
    
    We finish by making the essentially optimal choice $Q = X^{1/2}$, giving the claimed bound after observing that the primes $<\sqrt X$ make a negligible contribution. 
\end{proof}

\begin{lemma}\label{lemma:goodprimes} In each of the cases below, there exists a function $L(q)\to\infty$ (that may depend on the sequence $k_q$), such that:
    \begin{enumerate}
    \item[(1)] For a sequence $k_q\to\infty$ such that $k_q = q^{o(1)}$,
        given any positive integer $d$ and for all large enough (in terms of $d$)
        primes $q$, any non-zero polynomial $f$ of degree at most $d$ with coefficients bounded
        by $L(q)^{d+1}$, satisfies $f(k_q)\not\equiv 0\pmod q$. 
    \item[(2)]
        For a sequence $k_q\to\infty$ such that
        $k_q < \psi(q)$ for an increasing function $\psi$ with $\psi(X) = o(\pi(X))$, there exists a subset of the primes $\cP$ such that:
        \begin{itemize}
            \item $\#(\cP\cap [1,X]) = \pi(X)(1 - o(1))$ for all $X$,
            \item for every positive integer $d$ 
        there are only finitely many (in terms of $d$) $q \in \cP$ such that there
        exists a non-zero polynomial $f$ of degree at most $d$ with coefficients bounded by $L(q)^{d+1}$
        with $f(k_q) \equiv 0 \mod q$.
        \end{itemize}
    \item[(3)] For an irrational number $\theta$ and $\varepsilon > 0$, there exists a
    subset $\cP_\theta$ of the primes such that
    \begin{itemize}
        \item $\#(\cP_\theta \cap [1,X]) > \pi(X) - O_{\varepsilon}(X^{1/2 + \varepsilon})$ for all $X$,
        \item for every positive integer $d$, there are only finitely many (in terms of $d$)
            $q \in \cP_\theta$ such that there exists a non-zero
            polynomial $f$ of degree at most $d$ with coefficients bounded by
            $L(q)^{d+1}$ 
        with $f(\floor{q\theta}) \equiv 0 \mod q$.
    \end{itemize}
\end{enumerate}
\end{lemma}
\begin{proof}
    For part (1), we may choose $L(q) \leq \log E(q)\to\infty$ where $E(X)$ is supplied by Lemma \ref{lemma:fpalpha}, while also ensuring $E(q)\leq k_q$ (by shrinking $E$ if necessary). Note that since $k_q = q^{o(1)}$, if $f$ is a non-zero polynomial of
        degree at most $d$ with coefficients bounded by $L(q)^{d+1}$ such that
        $f(k_q)\equiv 0\pmod q$, then in fact $f(k_q)=0$ for large enough $q$. Each of the
        polynomials with coefficients bounded by $L(q)^{d+1}$ has at most $d$
        zeros, all of which are $\ll L(q)^{d+1}$ by Cauchy's bound. The growth of $L(q)$
        ensures that for only finitely many primes $q$ it is possible that $k_q$
        is a zero of such a polynomial, since $k_q\gg L(q)^{d+1}$ for large enough
        primes $q$. 

    For part (2), we consider the primes between $X$ and $2X$. We may choose the function $L$ to be constant on dyadic intervals.
        There are $\ll(L(X))^{O((d+1)^2)}$ non-zero polynomials of degree at most $d$ with coefficients bounded by $L(X)^{d+1}$, so by ensuring that $L(X)\leq \log E(X)$, Lemma
        \ref{lemma:fpalpha} implies there are at most $O((d+1)^2 \psi(2X)L(X)^{O((d+1)^2)}\log L(X))$
        primes $q$ in this range such that $k_q$ can be a
        root of one of these polynomials. Let $\cB_d(X,2X)$ be this set of
        primes.

        Now we take $\cB(X,2X) = \bigcup_{d=0}^{\log L(X)} \cB_d(X, 2X)$,
        say, so that
        \[
            \#\cB(X,2X) \ll \psi(2X) L(X)^{O((\log L(X))^2) + \varepsilon}.
        \]
     Additionally ensure that  $L(X)\to\infty$ while $L(X) \leq \log(\pi(X)/\psi(2X))^{1/3}$ and define $\cB =
        \bigcup_{j=0}^\infty \cB(2^j, 2^{j+1})$. Then we take $\cP$ to be all
        of the primes not in $\cB$.

    For part (3) we choose $L(q)\leq \log E_{\theta}(q)$ constant on dyadic intervals, with the choice of $E_{\theta}$ from Lemma \ref{lemma:fpalpha}, and we again consider the
    primes between $X$ and $2X$. There are $\ll (\log E_{\theta}(X))^{(d+1)^2}$
    non-zero polynomials of degree  at most $d$ with coefficients bounded by $(\log E_{\theta}(X))^{d+1}$,
    so by Lemma \ref{lemma:fpalpha} there are at most
    $O((d+1)^2 X^{1/2}(\log \log X)^{(d+1)^2})$ primes in this range that
    can be a root of one of these polynomials. Let $\cB_d(X,2X)$ be this
    set of primes.
    Now we take $\cB(X,2X) = \bigcup_{d=0}^{D(X)} \cB_d(X, 2X)$ with $(D(X)+1)\log\log E_{\theta}(X)\leq \log E_{\theta}(X)$ and $D(X)\to\infty$, say, so that
    $\#\cB(X,2X) \ll X^{1/2 + \varepsilon}$, and take
    $\cB = \bigcup_{j=0}^\infty \cB(2^j, 2^{j+1})$.
    Then we take $\cP_\theta$ to be all of the primes not in $\cB$.
\end{proof}

\begin{remark}
     We note that the argument in Lemma \ref{lemma:goodprimes} is  qualitative in nature and uses the conclusion of Lemma \ref{lemma:fpalpha}, without reference to the size of $k_q$. That is, the proof goes through for sequences $k_q\gg \pi(q)$, as long as one has an estimate of the form $o(\pi(X))$ on the number of primes $q\leq X$ such that $f(k_q)\equiv 0\pmod q$ for each \emph{fixed} polynomial $f$, with a reasonable explicit dependence on $f$. This is likely to be true for \emph{all} sequences satisfying $k_q = o(q)$ by a reasonably explicit quantitative equidistribution statement on the roots modulo prime $q$ of a fixed polynomial. However, it is a difficult open problem to establish such a statement for any polynomial $f$ with $\deg f\ge 3.$
    If equidistribution of roots modulo $q$ of, say, the polynomial  $f(x)=x^3+2$ fails, then we may have to exclude a positive proportion of primes for the sequence $k_q$ given by solutions to $k_q^3+2\equiv 0\pmod q$ with $k_q = o(q).$ Such pathological solutions would make a significant contribution to the finite moments in the proof of Theorem \ref{thm:random-process} which would change the distribution. We leave the details to the interested reader.   
\end{remark}

\section{\text{log}-integrability}

We now turn our attention to proving Theorem \ref{thm5}, as well as Corollary \ref{mahler-quant}. As we discussed in Section \ref{roadmap}, it is of key importance to us to understand the size of $$\int_{\alpha}^{\beta} \log\lvert f(t)\rvert\mathds1_{\lvert f(t)\rvert<\varepsilon}\ud t$$ uniformly in smooth functions $f:[\alpha, \beta]\to\mathbb R$, when $\varepsilon$ is small. 
In this section, we are going to prove the following key result.

\begin{proposition}\label{logThm}
    Let $\Delta = \sum_{i=0}^k a_i \partial^i$ be a linear differential operator of order $k\geq 1$ with constant real coefficients, and suppose that $(\Delta f)(t)\geq 1$ for all $t\in[\alpha, \beta]$ for real-valued $f\in\mathscr C^k[\alpha, \beta]$. Then for small enough $\varepsilon>0$, we have
    \[
    \int_{\alpha}^{\beta} \log\lvert f(t)\rvert \mathds1_{\lvert f(t)\rvert<\varepsilon}\ud t \ll \varepsilon^{\frac{1}{2k}},
    \]
    where the implied constant is uniform in $f$, but may depend on $\Delta$, $\alpha$ and $\beta$. 
\end{proposition}

Proposition \ref{logThm} will follow by combining the following two lemmas. The first lemma bounds the number of zeros of $f$ in terms of the number of zeros of $\Delta f$, whereas the second guarantees that $f$ does not spend too much time near $0$. 

\begin{lemma}\label{prop1}
    Let $\Delta = \sum_{i=0}^k a_i \partial^i$ be a linear differential operator of order $k$ with constant real coefficients. Suppose that $\Delta f$ has at most $m$ zeros in $[\alpha, \beta]$ for $f\in\mathscr C^{k}[\alpha,\beta]$. Then $f$ has $\leq m + O_{\Delta, \alpha, \beta}(1)$ zeros in $[\alpha, \beta]$, where the implied constant is uniform in $f$.
\end{lemma}

\begin{lemma}\label{prop2}
    Let $\Delta = \sum_{i=0}^k a_i \partial^i$ be a linear differential operator of order $k$ with constant real coefficients, and suppose that $(\Delta f)(t)\geq 1$ for all $t\in [\alpha, \beta]$ for $f\in\mathscr C^{k}[\alpha,\beta]$. Then for small enough $\varepsilon>0$ we have
    \[
        \mu\left(\{t\in[\alpha, \beta] : \lvert f(t)\rvert<\varepsilon \} \right)\ll \varepsilon^{\frac{1}{k}},
    \]
    where $\mu$ denotes the standard Lebesgue measure. 
\end{lemma}

\begin{proof}[Proof of Proposition \ref{logThm}, assuming Lemmas \ref{prop1} and \ref{prop2}]
    By Lemma \ref{prop1} and the condition that $(\Delta f)(t)\geq 1$ for
    all $t\in [\alpha, \beta]$, $f$ has at most $T$ zeros, for some finite
    number $T$ (independent of $f$).

Now decompose $(0, \varepsilon]$ into dyadic intervals and use Lemma
\ref{prop2} in each interval, replacing $\lvert f\rvert$ with the lower bound
of the interval, namely 
    \[
        \int\log\lvert f\rvert\mathds1_{\lvert f\rvert<\varepsilon}\ll \sum_{l=1}^\infty \int\log\lvert f\rvert\mathds1_{\lvert f\rvert\in\left[\frac{\varepsilon}{2^l}, \frac{\varepsilon}{2^{l-1}}\right)}\ll \sum_{l=1}^\infty \log \frac{2^{l-1}}{\varepsilon
        }\cdot(\varepsilon/2^{l})^{1/k}\ll \varepsilon^{1/k}\log\frac{1}{\varepsilon}\ll \varepsilon^{\frac{1}{2k}}.
    \]
This finishes the proof. \end{proof}

Now we prove Lemmas \ref{prop1} and \ref{prop2}. 

\begin{proof}[Proof of Lemma \ref{prop1}] We may assume that the differential operator $\Delta$ is monic, in which case we factorize it over $\mathbb R$ as follows: 
\[
    \Delta = \prod_{i=1}^{k_1}(\partial^2 - a_i\partial + b_i) \cdot\prod_{j=1}^{k_2}(\partial -c_j),
\]
where $a_i^2 <4b_i$ for each $1\leq i\leq k_1$. By induction, it suffices to show that both operators $\partial - c$ and $\partial^2 -a\partial + b$ do not significantly decrease the number of zeros. Indeed, $\partial - c$ does not decrease the number of zeros by more than one by Rolle's theorem applied to $e^{-cx}g(x)$, for
\[
(\partial - c)g = g'-cg = \frac{\partial}{\partial x}(e^{-cx} g(x) )e^{cx}.
\]
It remains to show the same holds for the operator $\partial^2 - a\partial + b$. Note that 
\[
    (\partial^2 - a\partial + b)f = u_2\partial\left(u_1\partial(uf) \right),
\]
where 
\[
    u_2(x) = e^{ax}u(x), u_1(x) = e^{-ax}u^{-2}(x),\text{ and } u(x) = \frac{e^{-\frac{a}{2}x}}{\cos\left(\frac{\sqrt{4b-a^2}}{2}x\right)}.
\]
The claim follows by applying Rolle's theorem twice in each open subinterval of $[\alpha, \beta]$ (depending on $a, b$) where $u$ has constant sign. In each such subinterval, the number of zeros does not decrease by more than two. It is clear that the number of such subintervals only depends on $\Delta$ and not on $f$, hence the conclusion follows. \end{proof}

\begin{proof}[Proof of Lemma \ref{prop2}] Let us start with a standard technical claim.
    There exists a smooth function $\phi : [0, 1]\to [0, 1]$, with the following properties: \begin{itemize}
        \item[(i)] $\phi(t) = 1$ for $t\in[\frac{1}{3}, \frac23]$;
        \item[(ii)] $\phi^{(k)}(0) = \phi^{(k)}(1) = 0$ for all $k=0, 1, 2, \ldots$;
        \item[(iii)] $\phi^{(k)}$ is bounded for every $k=0, 1, 2, \ldots$.
    \end{itemize}
    Indeed, define
\begin{equation}\label{eq:rho}
\beta(t)=
\begin{cases}
e^{-1/t} & t>0,\\[2mm]
0 & t\le 0,
\end{cases}
\qquad
\rho(t)=\frac{\beta(t)}{\beta(t)+\beta(1-t)},
\end{equation}
and set
\[
\phi(t)=\rho\left(3t\right)\rho\left(3(1-t)\right),\qquad t\in [0, 1].
\]
Note that $\phi$ is indeed a smooth function, supported on $[0, 1]$ which is equal to $1$ for $x\in [\frac 13, \frac 23]$. Moreover, all derivatives of $\phi$ vanish at $0$ and $1$ and are bounded. 

By Lemma \ref{prop1}, the functions $f-\varepsilon, f, f+\varepsilon$ each have a finite number of zeros (uniformly in $f$). For each such zero $x$ or $x\in\{\alpha, \beta\}$, let $I_x=[x^-, x^+]$ be the interval containing $x$ such that $\lvert f(t)\rvert\leq\varepsilon$ for each $t\in I_x$. Such intervals cover the subset of the domain where $\abs f<\varepsilon$. Let $\varphi$ be a smooth function, approximating $\mathds1_{I_x}$, with the property that $\varphi(x^-) = \varphi(x^+)=\varphi'(x^-)=\varphi'(x^+) = \cdots = \varphi^{(k)}(x^-) = \varphi^{(k)}(x^+) = 0$ (take $\varphi(x) = \phi((x-x^{-})/(x^+ - x^-))$). Note that 
$\int_{x^-}^{x^+} f^{(l)}\varphi = (-1)^l\int_{x^-}^{x^+} f\varphi^{(l)}$ by integration by parts and vanishing boundary terms and consequently  
\[
    \int_{x^-}^{x^+} \varphi(t)(\Delta f)(t)\ud t = \int_{x^-}^{x^+} f(t)(\widetilde{\Delta}\varphi)(t)\ud t,
\]
where $\widetilde{\Delta} = \sum_{i=0}^k (-1)^i a_i\partial^i$. By definition of $\varphi$ and the condition that $\Delta f\geq 1$, we have 
\[
    \int_{I_x} \varphi(t)(\Delta f)(t)\ud t\gg (x^+ -x^-),
\]
whereas
\[
    \int_{I_x} f(t)(\widetilde{\Delta}\varphi )(t)\ud t\ll \frac{\varepsilon}{(x^+-x^-)^{k-1}},
\]
since $\lvert f\rvert\leq\varepsilon$ on $I_x$, and $\varphi^{(i)}\ll (x^+-x^-)^{-i}$ by property (iii) and the chain rule. It follows that $x^+ - x^-\ll \varepsilon^{1/k}$. Given that there are finitely many values for $x$ (uniformly in $f$), this finishes the proof of the lemma.\end{proof}

\section{Proofs of Theorem \ref{thm5} and Corollary \ref{mahler-quant}}

\subsection{Proof of Theorem \ref{thm5}} 

The main difficulty in proving Theorem \ref{thm5} is that the functional on $\mathscr{C}[0,1]$ defined by $\ell(f)=\int_{0}^1\log \abs{f(t)}\ud t$ is not continuous, so our previous distributional result does not directly apply.
We instead consider 
$$\ell(f) = \ell_{\varepsilon}(f) + \int_{0}^{1}\log \abs{f(t)}(1-w_{\varepsilon}(\lvert f(t)\rvert))\ud t,$$ where
$$\ell_{\varepsilon}(f)=\int_{0}^{1}\log \abs{f(t)}w_{\varepsilon}(\lvert f(t)\rvert)\ud t.$$ Here $w_{\varepsilon}(t)=\rho\left(\frac{t}{\varepsilon} - 1\right)$ (recall the function $\rho$ in \eqref{eq:rho}) is a smoothed minorant of $\mathds 1_{t\geq\varepsilon}$, and $\ell_{\varepsilon}$ is a continuous functional on $\mathscr{C}[0,1]$. Roughly speaking, Theorem \ref{thm5} will follow from applying Theorem \ref{thm:random-quadratic-process} to a bounded version of $\ell_{\varepsilon}(f)$, and Proposition \ref{logThm} to $\ell(f) - \ell_{\varepsilon}(f)$ using the fact that $1-w_{\varepsilon}(t)$ is majorized by $\mathds 1_{t<2\varepsilon}$, where $f=F_{k, \chi_q, \alpha, \beta}$ with $\alpha=0.2$ and $\beta =1.1$. Indeed, considering a bounded version will be enough, for we will show that $\ell_{\varepsilon}(F_{k, \chi_q, \alpha, \beta})$ takes large values with small probability; see Lemma \ref{lemma:logLarge} below.

We start with the following lemma to handle $\ell(f) - \ell_{\varepsilon}(f)$. 
\begin{lemma}\label{lemmaLogSmall}
    For $\alpha=0.2$ and $\beta = 1.1$, for all large enough primes $q$ and all
    $k \in \{0, 1, \ldots, q-1\}$ we have
\[
    \int_0^1 \log \abs{F_{k,\chi_q,\alpha,\beta}(t)}\mathds{1}_{\abs{F_{k,\chi_q,\alpha,\beta}(t)}
        < \varepsilon} \ud t \ll \varepsilon^{1/6}
\]
    as $\varepsilon \rightarrow 0$, where $\chi_q$ is the quadratic character mod $q$.
\end{lemma}
\begin{proof}
     Recall that for a Dirichlet character $\chi$ mod $q$ and for $\alpha,\beta$
    such that $\alpha q$ and $(\alpha + \beta)q$ are not integers,
\[
    F_{k,\chi,\alpha,\beta}(t)
        = e(\alpha t)\frac{\tau(\chi)}{2\pi i q^{1/2}}\sum_{l \in \Z}
    \frac{e(\alpha l) (e(\beta(l + t)) - 1)}
    {l + t}
    \chibar(k-l).
\]

The only way we know how to show that this integral is well-behaved is via some
explicit numerical calculation that considers all possible values of $\chi_q(k
- l)$ for small $l$. The (ad-hoc) arguments used by \cite{KLM-fekete} for Fekete polynomials take advantage of two features
that we do not have at our disposal: they were able to deal with real-valued
functions using  the intermediate value theorem and most importantly, the derivatives of their series were absolutely convergent,
making the numerical computation (separation from zero value) feasible.

To get around these difficulties, we instead consider a linear combination
of derivatives of our function. Write
\[
    A(t) =
    \sum_{l \in \Z}
    \frac{e(\alpha l) (e(\beta(l + t)) - 1)}
    {l + t }
    \chi_q(k-l).
\]
Then
\[
    4\pi^2\beta^2 A'(t) + A'''(t) =
    \sum_{l \in \Z} \frac{\chi_q(k-l)}{(l+t)^4} a(l,t),
\]
where
\begin{multline*}
    a(l,t) = 
    \bigg[ 12 \pi \beta(l + t)e(\beta t)  e((\alpha + \beta)l) \left(\pi \beta (l+t) + i\right)
        \\
    - e(\alpha l)(e(\beta(l+t))-1)\left(6 + 4\pi^2\beta^2(l+t)^2\right) \bigg].
\end{multline*}
(Here $a(l,t)/(l + t)^4$ becomes $-4 \pi^4 \beta^4$ when $l = t = 0$
and $-4\pi^4\beta^4e(-\alpha)$ when $l=-1$ and $t = 1$.)
This is now something that is given by an absolutely convergent series, whose convergence is uniform in $t$, which implies that $A$ is (at least) three-times differentiable and that the linear combination of derivatives in consideration is indeed equal to the absolutely convergent series.
In particular, if we limit the sum to $l \in [-M, M-1]$, say, then the
truncation error will be at most
\[
\left(40 \pi^2 \beta^2 \right)\left(\zeta(2)
- \sum_{m < M}\frac{1}{m^2}\right)
+ 24 \pi \beta \left(\zeta(3) - \sum_{m < M}\frac{1}{m^3}\right)
+ 24 \left(\zeta(4) - \sum_{m < M}\frac{1}{m^4}\right).
\]

When $k \not \in \{0, q-1\}$, we can take $M = 2$ so that
\[
    \abs{4\pi^2\beta^2 A'(t) + A'''(t) - 
    \sum_{l = -2}^1 \frac{\chi_q(k-l)}{(l+t)^4} a(l,t)}
    < 326.9.
\]
There are 32 different possible values of the sequence
$\chi_q(k-1), \chi_q(k), \chi_q(k+1), \chi_q(k+2)$ (the first or the last term might
be zero, but not the middle two), and we check that for each possible choice
we can cover the interval $[0,1]$ by four separate intervals such that
on each interval, either $\abs{\real(4\pi^2\beta^2 A'(t) + A'''(t))} > 1$
or $\abs{\imag(4\pi^2\beta^2 A'(t) + A'''(t))} > 1$ across the entire
interval.

The terms with $l = -1$ and $l=0$ are the largest summands, so when one
of them vanishes we have to work a bit harder. Now in order to make the
computation tractable we use the symmetry and multiplicativity of $\chi_q$.
In the worst case, when $k=0$ and $\chi_q$ is an even character, we take $M =
28$, so that the error bound is 17.5. Using multiplicativity, we only need
to check 512 different possibilities for the sequence $\chi_q(-28), \chi_q(-27),
\ldots, \chi_q(27)$, rather than the $2^{56}$ possibilities we would have to
check naively.

In any case, we find that we can always cover $[0,1]$ by five intervals such that
on each interval, either $\abs{\real(4\pi^2\beta^2 A'(t) + A'''(t))} > 1$ or
$\abs{\imag(4\pi^2\beta^2 A'(t) + A'''(t))} > 1$ across the entire interval.

We are now in a position to appeal to Proposition \ref{logThm}. We have
\begin{equation*}
\int_0^1 \log \abs{A(t)}\mathds{1}_{\abs{A(t)} < 2 \pi \varepsilon} \ud t 
        = \sum_{j} \int_{I_j} \log \abs{A(t)}\mathds{1}_{\abs{A(t)} < 2 \pi \varepsilon} \ud t
\end{equation*}
for any disjoint collection of intervals $I_j$ that cover $[0,1]$, and when
$\varepsilon$ is small enough
\begin{multline*}
    \sum_{j} \abs{\int_{I_j} \log \abs{A(t)}\mathds{1}_{\abs{A(t)} < 2 \pi \varepsilon} \ud t}
\le \\
\sum_{j} \min \left(
    \abs{\int_{I_j} \log \abs{\real(A(t))}\mathds{1}_{\abs{\real(A(t))} < 2 \pi \varepsilon} \ud t},
    \abs{\int_{I_j} \log \abs{\imag(A(t))}\mathds{1}_{\abs{\imag(A(t))} < 2 \pi \varepsilon} \ud t}\right).
\end{multline*}
By the computation described above, we can split the interval $[0,1]$ into a
finite union of disjoint intervals (independent of $k$ and $\chi_q$) such
that on each interval, either the real part or the imaginary part of $A(t)$
satisfies the given linear differential inequality, and hence
\begin{multline*}
\min \left(
    \abs{\int_{I_j} \log \abs{\real(A(t))}\mathds{1}_{\abs{\real(A(t))} < 2 \pi \varepsilon} \ud t},
    \abs{\int_{I_j} \log \abs{\imag(A(t))}\mathds{1}_{\abs{\imag(A(t))} < 2 \pi \varepsilon} \ud t}\right)
    \ll
    \varepsilon^{1/6}
\end{multline*}
as $\varepsilon \rightarrow 0$ by Proposition \ref{logThm}, and the same bound
holds for the sum over all of the intervals.
\end{proof}

\begin{remark}
    In principle, it should be possible to carry out this procedure
    for any parameters $\alpha$ and $\beta$ as long as $4\pi^2\beta^2 A'(t) +
    A'''(t)$ does not vanish on $[0,1]$ for any realization of the
    coefficients. However, we are somewhat lucky here; there are
    certain values of $\alpha$ and $\beta$ where this quantity can be very
    small, and perhaps it may even vanish. In this case, it
    is possible to verify higher order differential inequalities to
    accomplish our goals. This will be explored in future work.
\end{remark}

Now we have 
\begin{multline}\label{eq:smallandlarge}
     \int_0^1 \log\abs{ \frac{1}{(1.1q)^{1/2}} S(\chi_q, 0.2, 1.1, \theta)}\ud\theta = \\  \mathbb E \int_0^1 \log\left\lvert\frac{1}{1.1^{1/2}}G_{q, 0.2, 1.1}(t)\ \right\rvert(w_{\varepsilon}(\lvert G_{q, 0.2, 1.1}(t)\rvert) + (1-w_{\varepsilon}(\lvert G_{q, 0.2, 1.1}(t)\rvert)))\ud t.
\end{multline} 

\begin{lemma}\label{lemma:logLarge}
    For any small enough $\varepsilon>0$, we have 
    \begin{multline*}
        \mathbb E \int_0^1 \log\left\lvert\frac{1}{1.1^{1/2}}G_{q, 0.2, 1.1}(t)\ \right\rvert w_{\varepsilon}(\lvert G_{q, 0.2, 1.1}(t)\rvert)\ud t\longrightarrow \\  \mathbb E\int_0^1 \log \abs{\frac{1}{1.1^{1/2}} G_{0.2, 1.1}(t)}w_{\varepsilon}(\lvert G_{0.2, 1.1}(t)\rvert)\ud t.
    \end{multline*}
\end{lemma}

\begin{proof}
    Note that for every realization of $G_{q, 0.2, 1.1}$, we have 
    \begin{equation*}
        \left(\int_0^1 \log\left\lvert\frac{1}{1.1^{1/2}}G_{q, 0.2, 1.1}(t) \right\rvert w_{\varepsilon}(\lvert G_{q, 0.2, 1.1}(t)\rvert)\ud t\right)^2\ll_{\varepsilon} \int_0^1 \left\lvert G_{q, 0.2, 1.1}(t) \right\rvert^2 \ud t
    \end{equation*}
    since Mahler measure is bounded above by the $L_2$ norm, whereas for the lower bound it suffices to take the implied constant to be $(\log(1/\varepsilon)/\varepsilon)^2$ since the integrand is supported on $\lvert G_{q, 0.2, 1.1}(t)\rvert\gg \varepsilon$.
    Taking expectation of both sides of the above inequality, considering that $\varepsilon>0$ is fixed, we obtain 
    \begin{equation*}
        \mathbb E\left(\int_0^1 \log\left\lvert\frac{1}{1.1^{1/2}}G_{q, 0.2, 1.1}(t)\ \right\rvert w_{\varepsilon}(\lvert G_{q, 0.2, 1.1}(t)\rvert)\ud t\right)^2\ll 1,
    \end{equation*} since
    \begin{align*}
        \mathbb E\int_0^1 \left\lvert G_{q, 0.2, 1.1}(t) \right\rvert^2 \ud t & = \frac{1}{q}\int_0^1 \sum_{\alpha q<n_1, n_2\leq (\alpha+\beta)q} \left( \frac{n_1n_2}{q}\right) \mathbb Ee\left(\frac{k+t}{q}(n_1-n_2) \right)\ud t \\ & = \frac{1}{q}\int_0^1 \sum_{\substack{\alpha q <n_1, n_2\leq (\alpha+\beta)q \\ q\mid n_1-n_2 \\ q\nmid n_1}} e\left(\frac{t(n_1-n_2)}{q} \right)\ud t
        \\ & \ll 1,
    \end{align*}
    as only the terms $n_1=n_2\neq 0\pmod q$ give a non-zero contribution, and there are $\ll \beta q$ such terms. Markov's inequality gives 
    \begin{equation*}
        \mathbb P \left(\int_0^1 \log\left\lvert\frac{1}{1.1^{1/2}}G_{q, 0.2, 1.1}(t) \right\rvert w_{\varepsilon}(\lvert G_{q, 0.2, 1.1}(t)\rvert)\ud t\geq M  \right)\ll \frac{1}{M^2}.
    \end{equation*}
    Therefore for any large $M>0$ we have 
    \begin{multline*}
        \mathbb E \int_0^1 \log\left\lvert\frac{1}{1.1^{1/2}}G_{q, 0.2, 1.1}(t)\right\rvert w_{\varepsilon}(\lvert G_{q, 0.2, 1.1}(t)\rvert)\ud t 
        \\ = \mathbb E \min\left\{\int_0^1 \log\left\lvert\frac{1}{1.1^{1/2}}G_{q, 0.2, 1.1}(t) \right\rvert w_{\varepsilon}(\lvert G_{q, 0.2, 1.1}(t)\rvert)\ud t, M \right\} + O\left(\frac{1}{M} \right). 
    \end{multline*}
    Applying Theorem \ref{thm:random-quadratic-process} to the bounded continuous functional $$f\mapsto \min\left\{\int_0^1 \log\left\lvert\frac{1}{1.1^{1/2}}f(t) \right\rvert w_{\varepsilon}(\lvert f(t)\rvert)\ud t, M \right\}$$ and letting $M\to\infty$ (with fixed $\varepsilon>0$) together with an application of the monotone convergence theorem yields the desired result. 
\end{proof}
Next, returning to \eqref{eq:smallandlarge}, using Lemmas \ref{lemma:logLarge} and \ref{lemmaLogSmall} together with the fact that $1-w_{\varepsilon}(t)$ is majorized by $\mathds 1_{t<2\varepsilon}$ (note that $\log\lvert t\rvert$ is negative for small enough $t$), we have
\begin{multline*}
    \int_0^1 \log\abs{ \frac{1}{(1.1q)^{1/2}} S(\chi_q, 0.2, 1.1, \theta)}\ud\theta
        \longrightarrow 
        \\ \mathbb E\int_0^1 \log \abs{\frac{1}{1.1^{1/2}} G_{0.2, 1.1}(t)}w_{\varepsilon}(\lvert G_{0.2, 1.1}(t)\rvert)\ud t+O(\varepsilon^{1/6}).
\end{multline*}
Letting $\varepsilon\to 0$ and using the fact that $w_{\varepsilon}(t)\to 1$ for every $t>0$ monotonically as $\varepsilon\to 0$, we have 
\[
    \int_0^1 \log\abs{\frac{1}{(1.1q)^{1/2}} S(\chi_q, 0.2, 1.1, \theta)}\ud\theta
        \longrightarrow \mathbb E\int_0^1 \log \abs{\frac{1}{1.1^{1/2}} G_{0.2, 1.1}(t)}\ud t.
\]
The exponential of the latter expectation may be computed to be equal to $0.954\ldots$.

\subsection{Proof of Corollary \ref{mahler-quant}}
We follow exactly the same steps as in the previous subsection. Let 
\[
    H^{\pm}_{k, q}(t) = F_{k, \chi_q, 0.2, 1.1}(t) \pm \frac{e(t)}{q^{1/2}},
\]
and let $H^{\pm}_q(t)$ be the random process obtained by choosing $k\pmod q$ uniformly at random. First, we have the following lemma as before.

\begin{lemma}\label{lemmaLogSmall2}
    For $\alpha=0.2$ and $\beta = 1.1$, for all large enough primes $q$ and all
    $k \in \{0, 1, \ldots, q-1\}$ we have
\[
    \int_0^1 \log \abs{H_{k, q}^{\pm}(t)}\mathds{1}_{\abs{H_{k, q}^{\pm}(t)}
        < \varepsilon} \ud t \ll \varepsilon^{1/6}
\]
    as $\varepsilon \rightarrow 0$, where $\chi_q$ is the quadratic character mod $q$.
\end{lemma}

\begin{proof}
The proof is identical to the proof of Lemma \ref{lemmaLogSmall}, provided we can show the additional term $\pm e(t)/q^{1/2}$ does not affect the lower bounds in our differential inequalities significantly. In fact, we will show that for the same five intervals covering $[0, 1]$ as in the previous subsection, in each interval either the real part or the imaginary part of (the appropriately normalized) $H^{\pm}_{k, q}$ satisfies exactly the same linear differential inequality. It is clear that it suffices to show that the first three derivatives (in $t$) of $y(t) = e(t)/q^{1/2}$ are small. Indeed, we uniformly have $y^{(l)}(t) = O_l(1/q^{1/2})$ for all $l\in\mathbb N$. We then apply Proposition \ref{logThm} and finish in exactly the same way as before.
\end{proof}
Note that for the continuous functional $\ell_{\varepsilon}$, the additional term $e(t)/q^{1/2}$ converges to $0$ uniformly in $t$, therefore invoking Lemma \ref{lemma:logLarge}, we  have 
$$
\lim_{q\to\infty}\mathbb E\int_{0}^{1}\log |{H}^{\pm}_{ q}(t)|w_{\varepsilon}(\lvert{H}^{\pm}_{ q}(t)\rvert)\ud t=\mathbb{E}\int_{0}^{1}\log |{G_{0.2, 1.1}}(t)|w_{\varepsilon}(|{G_{0.2, 1.1}}(t)|)\ud t.
$$ We conclude in exactly the same way as in the previous subsection, i.e. we have 
\begin{equation}\label{eqH_pconv}
    \lim_{q\to\infty}\mathbb E\int_{0}^{1}\log \abs{\frac{1}{1.1^{1/2}}{H}^{\pm}_{ q}(t)}\ud t
        = \mathbb{E}\int_{0}^{1}\log \abs{\frac{1}{1.1^{1/2}}{G_{0.2, 1.1}}(t)}\ud t.
\end{equation}
To finish the proof of Corollary \ref{mahler-quant}, note that the left-hand side of \eqref{eqH_pconv} is the limit of the normalized logarithmic Mahler measures of the Littlewood polynomials $$\sum_{n=1}^{\lfloor 1.3q\rfloor - \lfloor 0.2q\rfloor}\chi(n+\lfloor 0.2q\rfloor)x^{n-1} \pm x^{q-\lfloor0.2q\rfloor-1}$$ (upon noticing that Mahler measure is invariant under multiplication by powers of $x$). 

\section{Proof of Theorem \ref{GSP_extremal}}\label{section:GS}

In this section, we turn to Theorem \ref{GSP_extremal}, motivated by the conjecture of G\"unther and Schmidt \cite{GuentherSchmidt2017}
on the minimum $L_{2k}$ norm of the Turyn polynomial
\[
 e(-at)S(\chi_q, a/q, 1, t) = \sum_{n\leq q} \left(\frac{n+a}{q}\right)e(nt).
\]
As discussed in the introduction, in \cite{GuentherSchmidt2017} it was shown
that for every $k\in\mathbb N$, there is a function $\phi_k:\mathbb R\to\mathbb
R$ such that
\[
    \lim_{q\to\infty} \frac{1}{\sqrt q}\|S( \chi_q, a/q, 1, \cdot )\|_{2k}=\phi_k(\alpha) 
\]
when $a/q\to \alpha$. By a similar computation as in \eqref{eq:Computation}, note that the above limit is equal to the $2k$-th root of $\lim_{q\to\infty} \mathbb E \int_0^1\left\lvert G_{q, \alpha, 1}(t) \right\rvert^{2k}\ud t$, where $G_{q, \alpha, 1}$ is the random process $k \mapsto \frac{q^{1/2}}{\tau(\chi_q)}F_{k, \chi_q, \alpha, 1}(t)$,
where the set $\{0,1, \ldots, q-1\}$ is equipped with the uniform measure. G\"unther and Schmidt give an exact expression
for $\phi_k(\alpha)^{2k}$, but this expression does not seem to be so easy
to analyse. From our random process we get an alternate description of $\phi_k$.
\begin{lemma}\label{lemma:phik}
    For every positive real number $k$, the limit 
    $$\lim_{q\to\infty} \mathbb E \int_0^1\left\lvert G_{q, \alpha, 1}(t) \right\rvert^{2k}\ud t$$ exists and is equal to
\begin{equation}\label{eq:phi_k}
 \frac{1}{(2\pi)^{2k}}
    \int_0^1
        \mathbb E\left\lvert(e(t)-1)\sum_{m\in\mathbb Z}
        \frac{e(m\alpha)}{m+t}\bY(m)\right\rvert^{2k}\ud t, 
\end{equation}
where the $\bY(m)$ are $\pm 1$-valued uniform and independent random variables.
\end{lemma}
\begin{proof}
    Let $G_{q,\alpha} = G_{q,\alpha,1}$.
For any large $M>0$, we have 
\begin{multline*}
\mathbb E \int_0^1\left\lvert G_{q, \alpha}(t) \right\rvert^{2k}\ud t = \mathbb E \min\left\{\int_0^1\left\lvert G_{q, \alpha}(t) \right\rvert^{2k}\ud t, M\right\}  \\ +  O\left(\sum_{l=0}^\infty 2^{l+1}M\mathbb P\left(2^lM\leq \int_0^1\left\lvert G_{q, \alpha}(t) \right\rvert^{2k}\ud t  \leq 2^{l+1}M\right) \right).
\end{multline*}
The probability in the error term is of course bounded by 
\[
    \mathbb P\left(\int_0^1\left\lvert G_{q, \alpha}(t) \right\rvert^{2k}\ud t \ge 2^lM  \right).
\]
Markov's inequality gives
\begin{align*}
    \mathbb P\left(\int_0^1\left\lvert G_{q, \alpha}(t) \right\rvert^{2k}\ud t \ge 2^lM \right)
    \leq \frac{\mathbb E \left(\int_0^1\left\lvert G_{q, \alpha}(t) \right\rvert^{2k}\ud t  \right)^2}{2^{2l}M^2}
    \ll_k \frac{1}{2^{2l}M^2}, 
\end{align*} 
since the inequalities of Cauchy-Schwarz and H\"older with $m\geq 4k$ an even positive integer (e.g. $m=2\lceil2k\rceil$) yield
\begin{equation*}
  \mathbb E \left(\int_0^1\left\lvert G_{q, \alpha}(t) \right\rvert^{2k}\ud t  \right)^2\leq \mathbb E\int_0^1 \lvert G_{q, \alpha}(t)\rvert^{4k}\ud t\ll \mathbb E \left(\int_0^1 \abs{G_{q, \alpha}(t)}^{m}\ud t \right)^{\frac{4k}{m}}\ll_k 1.
\end{equation*}
The last inequality follows, for instance, by the
result of G\"unther and Schmidt that $\mathbb E\int_0^1 \lvert G_{q,
\alpha}(t)\rvert^{m}\ud t$ converges to $\phi_{m/2}(\alpha)^{m}$ as $q\to\infty$. We conclude
that 
\begin{equation*}
    \mathbb E \int_0^1\left\lvert G_{q, \alpha}(t) \right\rvert^{2k}\ud t = \mathbb E \min\left\{\int_0^1\left\lvert G_{q, \alpha}(t) \right\rvert^{2k}\ud t, M\right\} + O\left(\frac{1}{M} \right).
\end{equation*}
Now, $f\mapsto \min\left\{\int_0^1\left\lvert f(t) \right\rvert^{2k}\ud t, M\right\}$
is a bounded continuous functional. Hence as $q\to\infty$, by Theorem
\ref{thm:random-quadratic-process} we have
\begin{multline*}
    \lim_{q\to\infty} \mathbb E \int_0^1\left\lvert G_{q, \alpha, 1}(t) \right\rvert^{2k}\ud t \\ = \mathbb E\min\left\{\frac{1}{(2\pi)^{2k}}
    \int_0^1
        \left\lvert(e(t)-1)\sum_{m\in\mathbb Z}
        \frac{e(m\alpha)}{m+t}\bY(m)\right\rvert^{2k}\ud t, M\right\} + O\left(\frac{1}{M} \right).
\end{multline*}
Letting $M\to\infty$ and applying the monotone convergence theorem yields \eqref{eq:phi_k}.
\end{proof}

We may now redefine $\phi_k(\alpha)$ to be equal to the $2k$-th root of the expression in \eqref{eq:phi_k}, for any real $k>0$. We begin with the proof of part (i) of Theorem \ref{GSP_extremal}. Lemma \ref{lemma:phik}
already establishes the existence of $\phi_k$.
Simply expanding the summation inside the integral in
\eqref{eq:phi_k}, and bringing the expectation inside, we find that
\begin{multline*}
    \phi_k(\alpha)^{2k} = \frac{1}{(2\pi)^{2k}}
    \int_0^1
        \lvert e(t)-1\rvert^{2k}
        \sum_{m_i, n_i\in\mathbb Z}\Bigg(
            \frac{e(\alpha(m_1+\ldots +m_k - n_1-\ldots-n_k))}
                {(m_1+t)\cdots (m_k+t)(n_1+t)\cdots(n_k+t)} \times \\
            \mathbb E\big[\bY(m_1)\cdots \bY(m_k)\bY(n_1)\cdots \bY(n_k)\big]\Bigg)\ud t.
\end{multline*}
The expectation $\mathbb E\big[\bY(m_1)\cdots \bY(m_k)\bY(n_1)\cdots
\bY(n_k)\big]$ is non-zero if and only if the $m_i, n_i$ are equal in pairs, in
which case it is equal to $1$. Thus for each term in the summation that does
not vanish,
the denominator $(m_1+t)\cdots (m_k+t)(n_1+t)\cdots(n_k+t)$ is a square, and in
particular is nonnegative, so that
\begin{multline*}
    \phi_k(\alpha)^{2k} \le \frac{1}{(2\pi)^{2k}}
    \int_0^1
        \lvert e(t)-1\rvert^{2k}
        \sum_{m_i, n_i\in\mathbb Z}\Bigg(
            \frac{1}
                {(m_1+t)\cdots (m_k+t)(n_1+t)\cdots(n_k+t)} \times \\
            \mathbb E\big[\bY(m_1)\cdots \bY(m_k)\bY(n_1)\cdots \bY(n_k)\big]\Bigg)\ud t.
\end{multline*}
This expression is precisely equal to $\phi_k(0)^{2k}$, yielding part (i) of
the theorem.

It seems difficult to make further use of this combinatorial expression,
so we take a different approach for parts (ii) and (iii). Throughout the proof, we will assume that $k$ is a positive integer -- indeed this is enough to take the limits over all positive real $k>0$ in parts (ii) and (iii), by monotonicity of $k\mapsto \| S(\chi_q, \alpha, 1, \cdot)\|_{2k}$ and the slow variation of $k\mapsto \log k$. Let
\[
    \begin{split}
    R_\theta(t)
    &= \Re\left[
        \left\lvert e(t)-1\right\rvert\left(\sum_{m\in\mathbb Z}
    \frac{e(m\alpha + \theta)}{m+t}\bY(m)\right)\right] \\
    &= 
    \left\lvert e(t)-1\right\rvert\left(\sum_{m\in\mathbb Z}
    \frac{\cos(2\pi m\alpha + 2\pi\theta)}{m+t}\bY(m)\right).
    \end{split}
\]
Note that 
\begin{equation*}
    \phi_{k}(\alpha)^{2k} =
    \frac{1}{(2\pi)^{2k}}
    \int_0^1
    \mathbb E \max_\theta R_\theta(t)^{2k}\ud t.
\end{equation*}
We emphasize that here the maximum over $\theta$ in the above integral depends both on $t$ and a realization of all the $\bY(m)$. We first reduce upper and lower bounds for $\phi_k(\alpha)$ to those for 
$$\max_{\theta} \int_0^1 \mathbb E R_{\theta}(t)^{2k}\ud t,$$ where now the maximum over $\theta$ does \emph{not} depend on $t$ or the realization of the $\bY(m)$. For the lower bound, note that clearly
\[
    \phi_{k}(\alpha)^{2k} =
    \frac{1}{(2\pi)^{2k}}
    \int_0^1
    \mathbb E \max_\theta R_\theta(t)^{2k}\ud t 
       \ge
    \frac{1}{(2\pi)^{2k}}
    \max_\theta
    \int_0^1
    \mathbb E R_\theta(t)^{2k}\ud t.
\]
For the upper
bound, note that for any large integer $N$ and any realization of the $\bY(m)$, using an elementary trigonometric inequality
we have 
\[
\begin{split}
    \max_\theta R_\theta(t)^{2k}
        &\le \sec(2\pi/N)^{2k} \max_{0 \le j < N} R_{j/N}(t)^{2k} \\
        &\le \sec(2\pi/N)^{2k} \sum_{j=0}^{N-1} R_{j/N}(t)^{2k}.
\end{split}
\]
Therefore 
\begin{align*}
    \int_0^1 \bE \max_{\theta} R_\theta(t)^{2k} \ud t&  \le 
    \sec(2\pi/N)^{2k} \sum_{j=0}^{N-1} \int_0^1 \bE R_{j/N}(t)^{2k}\ud t \\ &
        \le
        \sec(2\pi/N)^{2k} N \max_{\theta} \int_0^1 \bE R_\theta(t)^{2k}\ud t.
\end{align*}

Next, we reduce upper and lower bounds for $\max_{\theta} \int_0^1 \mathbb E R_{\theta}(t)^{2k}\ud t$ to those for 
$$
    \max_{\theta}
    \sum_{\substack{m=-M \\ m \ne 0}}^{M}
    \abs{\frac{\cos(2\pi m\alpha + 2\pi \theta)}{m}}$$ for a suitably chosen $M$ (in terms of $k$). 
For the lower bound,
we note that for any $\theta$ and some given value of $M$, the
probability that $\bY(m) \cos(2\pi m\alpha + 2\pi \theta)/(m+t)$
has constant sign for all $0 <\lvert m\rvert  \le M$ and $m\neq -1$ is at least $2^{-2M+1}$, and for
such a realization we have
\[
   \left\lvert\sum_{m\in\mathbb Z}
   \frac{\cos(2\pi m\alpha + 2\pi \theta)\bY(m)}{m+t}\right\rvert
   \ge 
   \sum_{\substack{m=- M\\ m \not\in\{-1, 0\}}}^{M}
   \abs{\frac{\cos(2\pi m\alpha + 2\pi \theta)}{m+t}}
\]
with probability at least $1/2$ (if the above inequality fails, changing $\bY(m)$ to $-\bY(m)$ for $\lvert m\rvert >M$ and $m=0, -1$ makes the inequality true). Now we note that
\[
   \sum_{\substack{m=-M \\ m \not\in \{-1, 0\}}}^{M}
   \abs{\frac{\cos(2\pi m\alpha + 2\pi \theta)}{m+t}}
   =
   \sum_{\substack{m=-M \\ m \ne 0}}^{M}
   \abs{\frac{\cos(2\pi m\alpha + 2\pi \theta)}{m}} + O(1).
\]
Thus
\begin{multline*}
    \left(\max_{\theta} \int_0^1 \bE R_\theta(t)^{2k} \ud t\right)^\frac{1}{2k}
    \\
    \ge
    \frac{1}{2^{M/k}}\left(
    \max_{\theta}
    \sum_{\substack{m=-M \\ m \ne 0}}^{M}
    \abs{\frac{\cos(2\pi m\alpha + 2\pi \theta)}{m}} + O(1)
    \right)
    \left(\int_0^1 \abs{e(t) - 1}^{2k} \ud t\right)^\frac{1}{2k}
    \\
    =
    \frac{1}{2^{M/k}}
    \binom{2k}{k}^\frac{1}{2k}
    \left(
    \max_{\theta}
    \sum_{\substack{m=-M \\ m \ne 0}}^{M}
    \abs{\frac{\cos(2\pi m\alpha + 2\pi \theta)}{m}} + O(1)
    \right) \\
    =
    \frac{1}{2^{M/k}}
    (2 + o(1))
    \left(
    \max_{\theta}
    \sum_{\substack{m=-M \\ m \ne 0}}^{M}
    \abs{\frac{\cos(2\pi m\alpha + 2\pi \theta)}{m}} + O(1)
    \right).
\end{multline*}

For the upper bound, from Minkowski's triangle inequality and the usual triangle inequality
we have
\begin{multline*}
    (\bE R_\theta(t)^{2k})^{1/2k}
    \le  \abs{e(t) - 1}\sum_{m=-M}^M
    \abs{\frac{\cos(2\pi m\alpha + 2\pi \theta)}{m + t}} \\
    +  \abs{e(t) - 1}\left(\bE \abs{
    \sum_{\abs{m} > M}
    \frac{\cos(2\pi m\alpha + 2\pi \theta)}{m + t}\bY(m)}^{2k}\right)^{1/2k}.
\end{multline*}
Further, from Khintchine's inequality, $\abs{e(t) - 1} \le 2$ and $\abs{e(t) - 1}/t \le 2\pi$
we have
\begin{multline*}
    (\bE R_\theta(t)^{2k})^{1/2k}
    \\ \le 
        4\pi + 
        2 \sum_{\substack{m=-M\\ m \not\in\{-1, 0\}}}^M
    \abs{\frac{\cos(2\pi m\alpha + 2\pi \theta)}{m + t}}
    + 2C_{2k}\left(
    \sum_{\abs{m} > M}\abs{\frac{\cos(2\pi m\alpha + 2\pi \theta)}{m + t}}^{2}\right)^{1/2},
\end{multline*}
where $C_{2k}$ may be chosen to be $\ll \sqrt{k}$; see \cite{haagerup-khintchine}. Taking $M=k$ we have
\[
    (\bE R_\theta(t)^{2k})^{1/2k}
    \le
        2 \sum_{\substack{m=-k \\ m \ne 0}}^k
    \abs{\frac{\cos(2\pi m\alpha + 2\pi \theta)}{m}} + O(1).
\]

Finally, we evaluate
\[
    \max_{\theta}
    \sum_{\substack{m=-M \\ m \ne 0}}^M
    \abs{\frac{\cos(2\pi m\alpha + 2\pi \theta)}{m}},
\]
considering $\alpha$ fixed while $M$ tends to infinity. When $\alpha = a/b$
we have
\[
\begin{split}
    \max_{\theta}
    \sum_{\substack{m=-M \\ m \ne 0}}^M
    \abs{\frac{\cos(2\pi ma/b + 2\pi \theta)}{m}}
    &= 
    \max_{\theta}
    \sum_{\substack{m=-\floor{M/b}}}^{\floor{M/b}}
    \sum_{\substack{c=0 \\ bm+c \ne 0 }}^{b-1}
    \abs{\frac{\cos(2\pi ca/b + 2\pi \theta)}{bm+c}}
    + O(1) \\
    &= 
    \max_{\theta}
    \sum_{\substack{m=-\floor{M/b}\\ m\neq 0}}^{\floor{M/b}}
    \frac{1}{|bm|}
    \sum_{\substack{c=0}}^{b-1}
    \abs{\cos(2\pi ca/b + 2\pi \theta)}
    + O(1) \\
    &= 
    \max_{\theta}
    \frac{1}{|b|}
    \sum_{\substack{m=-\floor{M/b}\\ m\neq 0}}^{\floor{M/b}}
    \frac{1}{|m|}
    \sum_{\substack{c=0}}^{b-1}
    \abs{\cos(2\pi ca/b + 2\pi \theta)}
    + O(1).
\end{split}
\]
Now, the quantity
\[
    \max_{\theta}
    \sum_{c=0}^{b-1}
    \abs{\cos(2\pi ca/b + 2\pi \theta)}
\]
clearly does not depend on $a$, provided that $(a,b)=1$. For odd $b$,
by $1/(2b)$-periodicity, we may assume that
$\lvert\theta\rvert\leq 1/(4b)$. Relabeling the signed residues, we have
\[
\begin{split}
\sum_{c=0}^{b-1} \lvert \cos(2\pi(c/b+\theta))\rvert
&= \cos(2\pi\theta)
+ \sum_{c=1}^{\frac{b-1}{2}}
\left(
\cos(\pi c/b+2\pi\theta)
+ \cos(-\pi c/b+2\pi\theta)
\right)
\\
&= \cos(2\pi\theta)
\left(1+2\sum_{c=1}^{\frac{b-1}{2}}\cos(\pi c/b)\right)
\\
&= \csc\left(\frac{\pi}{2b}\right)\cos(2\pi\theta).
\end{split}
\]
Doing a similar calculation for even $b$, we use $1/b$-periodicity.
If $b\equiv 2\pmod{4}$, we may assume that
$\lvert\theta\rvert\leq 1/(2b)$, and we have
\[
\begin{split}
\sum_{c=0}^{b-1} \lvert \cos(2\pi(c/b+\theta))\rvert
&= 2\sum_{c=-\frac{b-2}{4}}^{\frac{b-2}{4}}
\cos(2\pi(c/b+\theta))
\\
&= 2\cos(2\pi\theta)
\left(1+2\sum_{c=1}^{\frac{b-2}{4}}\cos(2\pi c/b)\right)
\\
&= 2\csc\left(\frac{\pi}{b}\right)\cos(2\pi\theta).
\end{split}
\]
If $b\equiv 0\pmod{4}$, we may assume that
$0\leq\theta\leq 1/b$, and we have
\[
\begin{split}
\sum_{c=0}^{b-1} \lvert \cos(2\pi(c/b+\theta))\rvert
&= 2\sum_{c=-\frac{b}{4}}^{\frac{b}{4}-1}
\cos(2\pi(c/b+\theta))
\\
&= 4\cos\left(2\pi\theta-\frac{\pi}{b}\right)
\sum_{c=1}^{\frac{b}{4}}\cos\left(\frac{(2c-1)\pi}{b}\right)
\\
&= 2\csc\left(\frac{\pi}{b}\right)
\cos\left(2\pi\theta-\frac{\pi}{b}\right).
\end{split}
\]
It becomes transparent that the maximum occurs
for $\theta=0$ when $b$ is odd or $b\equiv 2\pmod{4}$ and
$\theta=1/(2b)$ when $b\equiv 0\pmod{4}$ and that
it is equal to
\[
    C_b = \begin{cases}
        \csc(\pi/2b) & \textrm{ if } b \textrm{ is odd,} \\
        2\csc(\pi/b) & \textrm{ if } b \textrm{ is even.}
    \end{cases}
\]
Putting things together, we have
\begin{multline*}
    \frac{1}{2^{M/k}}\frac{2 + o(1)}{2\pi}\left(\frac{2C_b}{b} \log (M/b) + O(1)\right)
        \le \phi_k(a/b) \\
        \le \frac{2}{2\pi}
        (N\sec(2\pi/N)^{2k})^{1/2k}
        \left(\frac{2C_b}{b}  \log(k/b) + O(1)\right).
\end{multline*}
Now take $N=k$, $M = \floor{k/\log k}$ and let $k\to\infty$ to finish the proof of part (ii).

For irrational $\alpha$, first note that from summation by parts
\[
    \sum_{m=1}^M
    \abs{\frac{\cos(2\pi m\alpha + 2\pi \theta)}{m}}
    = \int_1^M \frac{1}{x^2} \left(\sum_{1 \le m \le x} \abs{\cos(2\pi m\alpha + 2\pi \theta)}\right) \ud x + O(1).
\]
For fixed irrational $\alpha$, from the uniform distribution of
$\{m \alpha\}$ we have
\[
\left(\sum_{1 \le m \le x} \abs{\cos(2\pi m\alpha + 2\pi \theta)}\right)
 = \frac{2}{\pi} x + o(x),
\]
uniformly in $\theta$ (as a shift does not change the discrepancy of the
sequence),
where
\[
    \frac{2}{\pi} = \int_0^1 \abs{\cos(2\pi x)} \ud x.
\]
The same holds for negative $m$, so
\[
    \sum_{\substack{m=-M \\ m \ne 0}}^M
    \abs{\frac{\cos(2\pi m\alpha + 2\pi \theta)}{m}}
    =
    \frac{4}{\pi}\int_{1}^M \frac{1 + o(1)}{x} \ud x
    = 
    \frac{4}{\pi}\log M(1+o(1)).
\]
Thus we find that
\[
    \max_{\theta}
    \sum_{\substack{m=-M \\ m \ne 0}}^M
    \abs{\frac{\cos(2\pi m\alpha + 2\pi \theta)}{m}}
    =
    \frac{4}{\pi} \log M + o(\log M),
\]
which gives part (iii) of the theorem.

\subsection{Small \texorpdfstring{$k$}{k}}
Using the expression given by G\"unther and Schmidt
\cite{GuentherSchmidt2017}*{Theorem 2.5}, one can compute that
\[
    \phi_8(x)^{16} = \frac{1}{212837625}
    \begin{cases}
        g_0(x) & \text{if } 0 \le x \le 1/8 \\
        g_1(x) & \text{if } 1/8 \le x \le 1/6 \\
        g_2(x) & \text{if } 1/6 \le x \le 1/4,
    \end{cases}
\]
where
\[
    \begin{split}
        g_0(x) = 
        &\, \, 30415597585031642480640  x^{14} - 50207186707872554680320 x^{13} \\
        &+ 39512718251975882833920  x^{12} - 19839240425139272417280 x^{11} \\
        &+ 7160618031823115059200  x^{10} - 1979491957237392998400  x^{9} \\
        &+ 434872982120232015360  x^{8} - 77567882310105707520  x^{7} \\
        &+ 11354765927079294720  x^{6} - 1366649252083315200  x^{5} \\
        &+ 134245773534032160  x^{4} - 10551750722472960  x^{3} \\
        &+ 637215707768640  x^{2} - 27120702248640 x + 643983856759
    \end{split},
\]
\[
    \begin{split}
        g_1(x) = 
    &\, \, 30415597585031642480640x^{14} - 54347864577634467840000x^{13} \\
    &+ 45559829304959049400320x^{12} - 23784915040170961797120x^{11} \\
    &+ 8661013482608416849920x^{10} - 2339023211620047912960x^{9} \\
    &+ 487721357130000913920 x^{8} - 80999616107553960960 x^{7} \\
    &+ 11001320628321703680 x^{6} - 1246972632231406080 x^{5} \\
    &+ 118754731912296480 x^{4} - 9343763663794560 x^{3} \\
    &+ 578058364403520 x^{2} - 25430554749600 x + 622343775439
\end{split},
\]
and
\[
    \begin{split}
    g_2(x) = 
    &\, \, 30415597585031642480640  x^{14} - 89872789029880804147200 x^{13} \\
    & + 123666115402944561807360 x^{12} - 104831188289515070423040 x^{11} \\
    & + 61066944318730117447680 x^{10} - 25825668757878357688320 x^{9} \\
    & + 8166954335613399406080  x^{8} - 1959485368580523463680 x^{7} \\
    & + 357982227596277123840 x^{6} - 49490574187638904320 x^{5} \\
    & + 5089686890984535840  x^{4} - 377098117788712320 x^{3} \\
    & + 19013610929913600 x^{2} - 584626429427040 x + 8340004663519
    \end{split}.
\]
Then we see that $\phi_8$ has a unique local minimum $x_0 \approx 0.2411345$
in this range, where we have $\phi_8(x_0) \approx 1.2493898$ in contrast
to $\phi_8(1/4) = (960901090937/27243216000)^{1/16} \approx 1.24943529$.
Similarly for $k=9,10,11$ there is a unique local minimum occurring at approximately
$0.233400, 0.228561$ and $0.225292$, respectively.

The computation of these polynomials seems to be exponentially expensive, so we
could compute a few more, but it is difficult to go much further.

\section*{Acknowledgements}

We are grateful to Yu-Chen Sun, Neo Tardy and Victor Wang for fruitful conversations, as well as to Sacha Mangerel for a careful reading of, and detailed feedback on, the third author's thesis, of which this paper was an integral part. We particularly thank Neo Tardy for pointing out an error in our previous argument regarding the G\"unther--Schmidt conjecture. B.Sh. is funded by a University of Bristol PhD scholarship.

\section*{AI Declaration}

The first two arXiv versions of this paper were written entirely independently of any AI tools. Following the communication with Neo Tardy regarding our error, a version of ChatGPT was used to numerically disprove the G\"unther--Schmidt conjecture (in April 2026), as well as to perform guided calculations that appear in Section 9 only; the ideas are due to the authors. Subsequently, ChatGPT 5.6 Sol and ChatGPT 6 Astra were used to identify minor errors throughout the paper, all of which have been reviewed and corrected by the authors.

\end{document}